\documentclass[12pt]{amsart}

\usepackage[utf8]{inputenc}
\usepackage{amsmath,amssymb,bbm,enumitem}
\usepackage{hyperref}
\usepackage{amsthm}
\usepackage{listings}[2013/08/05]
\usepackage[capitalize]{cleveref}
\usepackage[square,numbers,sort]{natbib}
\usepackage{xparse}
\usepackage{tikzit}
\usepackage{braids}

\hypersetup{pdfstartview=}

\def\clap#1{\hbox to 0pt{\hss#1\hss}}

\def\mathclap{\mathpalette\mathclapinternal}

\def\mathclapinternal#1#2{%
\clap{$\mathsurround=0pt#1{#2}$}}

\newcommand{\rhup}{\rightharpoonup}

\newcommand{\ket}[1]{|{#1}\rangle}

\def\BN{{\mathbb N}}

\def\BZ{{\mathbb Z}}
\def\BC{{\mathbb C}}

\def\B{{\mathcal B}}
\def\D{{\mathcal D}}

\def\C{{\mathcal C}}
\def\U{{\mathcal U}}
\def\Z{{\mathcal Z}}

\def\k{{\mathbbm k}}

\def\ot{\otimes}
\newcommand\inv{^{-1}}
\def\iff{\Leftrightarrow}

\newcommand{\cyc}[1]{\langle #1 \rangle}

\newcommand{\ip}[2]{\langle#1,#2\rangle}

\newcommand{\du}[1]{\mathbbm{k}^{#1}}
\newcommand{\duc}[1]{\mathbbm{k}^{#1\operatorname{co}}}

\newcommand\comm\curlyvee
\newcommand\cocomm\curlywedge

\newcommand{\com}[1]{_{(#1)}}

\newcommand{\BCh}[1]{\Hom(#1,\widehat{#1})}

\DeclareMathOperator{\VecG}{Vec}

\DeclareMathOperator{\image}{Im}  
\DeclareMathOperator{\Img}{\image}
\DeclareMathOperator{\Real}{Re}   
\DeclareMathOperator{\Hom}{Hom}   
\DeclareMathOperator{\Res}{Res}   
\DeclareMathOperator{\id}{id}     
\DeclareMathOperator{\sgn}{sgn}   
\DeclareMathOperator{\class}{class}
\DeclareMathOperator{\Aut}{Aut}   
\DeclareMathOperator{\Inn}{Inn}   
\DeclareMathOperator{\End}{End}   
\DeclareMathOperator{\Irr}{Irr}   
\DeclareMathOperator{\Rep}{Rep}		
\DeclareMathOperator{\co}{co}     
\DeclareMathOperator{\op}{op}			
\DeclareMathOperator{\cl}{cl}     

\theoremstyle{plain}
\newtheorem{thm}{Theorem}[section]
\newtheorem*{theorem*}{Theorem}
\newtheorem{cor}[thm]{Corollary}
\newtheorem{prop}[thm]{Proposition}
\newtheorem{lem}[thm]{Lemma}

\theoremstyle{definition}
\newtheorem{df}[thm]{Definition}
\newtheorem{example}[thm]{Example}

\theoremstyle{remark}
\newtheorem{rem}[thm]{Remark}

\crefname{lem}{Lemma}{Lemmas}
\crefname{thm}{Theorem}{Theorems}
\crefname{cor}{Corollary}{Corollaries}
\crefname{prop}{Proposition}{Propositions}
\crefname{example}{example}{examples}
\crefname{df}{Definition}{Definitions}
\crefname{equation}{equation}{equations}

\numberwithin{equation}{section}

\DeclareDocumentCommand{\morph}{ O{u} O{r} O{p} O{v} }{\begin{pmatrix} #1 & #2 \\ #3 & #4 \end{pmatrix}}
\DeclareDocumentCommand{\cmorph}{ O{u} O{r} O{p} O{v} }{(#1,#2,#3,#4)}
\DeclareDocumentCommand{\fqts}{ O{r} O{p} }{(1,#1,#2,0)}

\newcommand{\DDG}{\D(G)^{*\co}}
\DeclareDocumentCommand{\twocycle}{ O{u} O{r} O{p} O{v} O{G}}{\sum_{g,h,s,t\in #5} #2^*(h)e_s\# gt \ot e_h #1(e_g)\# #4(s)#3(e_t)}

\newlist{lemenum}{enumerate}{1}
\setlist[lemenum]{label=\roman*), ref=\textup{\thethm~(\roman*)}}
\crefalias{lemenumi}{lemma}

\title[Braid gauging applications]{Braid gaugings and categorical invariants}
\author{Marc Keilberg}
\email{keilberg@usc.edu}

\begin{document}
\begin{abstract}
We study the categorical notion of braid gauging and obtain its classical Hopf algebraic description. We demonstrate how braid gauging can provide new insights on certain categorical invariants, such as the fusion rules and the higher Frobenius-Schur indicators. The running example for the paper is the category $\Rep(\D(G))\cong \Z(\VecG_G)$, whose braid gaugings are studied in-depth.
\end{abstract}
\keywords{quasitriangular structures, group doubles, braidings, higher Frobenius-Schur indicators, Verlinde formula, modular categories}
\thanks{The author thanks Peter Schauenburg for many helpful discussions.}
\maketitle

\section{Introduction}\label{sec:intro}
A complete, group-theoretical description of the quasitriangular structures for $\D(G)$, the Drinfeld double of a finite group $G$, was obtained by the author in \citep{K3:QTS}. This is equivalently a description of all possible braidings with which the tensor category $\Rep(\D(G))$ of finite dimensional representations of $\D(G)$ can be endowed. \citet{Nik:braidings} has more recently obtained an alternative description from a more general categorical perspective.  While many of the classical applications of the classification from \citep{K3:QTS} were discussed therein---such as a description of the ribbon elements and a discussion of when there were isomorphisms of ribbon Hopf algebras---most of the applications to the categorical side, such as categorical invariants, were left open.  This paper arose out of the explorations of such applications.

Part of \citep{K3:QTS} was spent investigating and classifying what were called "central weak $R$-matrices" of $\D(G)$.  While these seemed to carry a lot of structure, and seemed closely related in form to the quasitriangular structures, it was suggested to the author that these had no categorical interpretation, and so the similarity was likely a coincidence of limited use outside of the very specific category under consideration. We show in \cref{sec:gauge} that these Hopf algebraic objects do, in fact, have a categorical interpretation as the braid gaugings of the category, as defined in \citep{Nik:braidings}.  This permits an easy translation of the results of \citep{K3:QTS} into a much more general categorical setting, and explains all of the apparent structure and relations to the quasitriangular structures.  We leverage this in \cref{sec:invariance} to establish the following invariance property:
\begin{theorem*}
  Let $G$ be a purely non-abelian group.  Then the pivotal category $\Rep(\D(G))$ has a unique structure as a braided tensor category, up to braided tensor equivalence.
\end{theorem*}

Of additional interest is the applications of braid gaugings to the study of invariants of modular fusion categories. Modular fusion admit two invertible matrices $S$ and $T$, called the modular data, which act on the complexification of the Grothendieck ring.  These yield a finite dimensional projective representation of $\text{SL}_2(\BZ)$. Such categories are ubiquitous in mathematical physics, details of which can be found in many of the references; \citep{BakKir:book} is particularly informative in this regard.  These invariants frequently either use braidings directly in their definitions, or admit formulas in terms of the modular data. The two we will single out are the higher Frobenius-Schur indicators \citep{KSZ2,NS07a} and the fusion coefficients, as expressed by the Verlinde formula (cf.~\citep{BakKir:book}). These are particularly noteworthy because they can be defined independently of the modular data, and so are invariant under a change of braiding.  A description of the braidings yielding a modular category then permits combining these formulas across many such choices, allowing many new identities and relations to be discovered. In this fashion we obtain new information about a category by studying its braidings.

The paper is otherwise structured as follows.  We review notation and background material in \cref{sec:prelims}.  Then in \cref{sec:gauge} we recall the notion of braid gauging from \citep{Nik:braidings}, and show how these are precisely a categorification of a certain collection of weak $R$-matrices \citep{Jiao:QTSwSP,Rad:QSHA}.  As a result we show that the "central weak $R$-matrices of $\D(G)$", as defined and classified in \citep{K3:QTS}, are precisely the braid gaugings of $\Rep(\D(G))$. This answers several questions from \citep{K3:QTS} concerning the structure of these objects, in particular their relationship to the quasitriangular structures.  We discuss the basic impact that braid gaugings have on certain categorical invariants in \cref{sec:links}. We then move to an in-depth study of the braidings and braid gaugings of $\Rep(\D(G))$ in the next two sections. \Cref{sec:chirality} is primarily dedicated to introducing additional terminology to facilitate this goal.  \Cref{sec:invariance} then establishes the aforementioned theorem that $\Rep(\D(G))$ has a unique structure as a braided tensor category whenever $G$ is purely non-abelian. The paper concludes by investigating new identities for the higher Frobenius-Schur indicators (\cref{sec:indicators}) and the fusion coefficients and $S$-matrix entries (\cref{sec:verlinde}) that are obtained from braid gaugings on a modular category, with a number of examples for the special case of $\Rep(\D(G))$.

\section{Preliminaries and Notation}\label{sec:prelims}
We begin by reviewing all of the background material and notation we will need.

We work over the field $\k=\BC$ of complex numbers, and define $U(1)$ to be the multiplicative group of norm 1 elements of $\BC$.  For the general theory of Hopf algebras, we refer the reader to \citep{Mon:HAAR}. All morphisms will be morphisms of Hopf algebras or groups, as appropriate, unless otherwise noted.  All unadorned tensor products are taken over $\k$.  We let $\tau$ denote the coordinate exchange map on pairs: $(h,k)\mapsto (k,h)$. On tensor products this would mean for $h\ot k\in H\ot K$ that $\tau(h\ot k)=k\ot h\in K\ot H$.

Many of the sets of morphisms we work with will in fact be abelian groups under convolution products.  So for the same reasons given for \citep[Definition 3.5]{K3:QTS}, we will adopt an additive notation for convolution products whenever convenient.  The convolution product of morphisms $u,v$ will be generally denoted by $u+v$, or by $u*v$ when the additive notation is otherwise inconvenient.  In the additive notation, $u-v$ denotes the convolution product $u+Sv$ of $u$ and $Sv$, where $S$ denotes the antipode. More generally, a minus sign on a morphism stands for the antipode. We will denote the $n$-fold composition of a function $f$ with itself, when defined, by $f^n$. Identity morphisms of groups or Hopf algebras will be denoted by 1, and trivial morphisms of groups or Hopf algebras will be denoted by 0.  The exceptions to this will be that the trivial character will be denoted $\varepsilon$, as will the counit of a Hopf algebra.

\subsection{Groups}
\begin{df}
  Let $G$ be a finite group.
  \begin{enumerate}
    \item $G$ is indecomposable if it has no non-trivial proper direct factors.
    \item $G$ is said to be purely non-abelian if it has no non-trivial abelian direct factors.
    \item The derived subgroup of $G$ is denoted $G'$.  This is the subgroup generated by the set of all commutators $[g,h] = g\inv h\inv g h = g\inv g^h$, where $g,h\in G$.
    \item We say $G$ is perfect if $G'=G$.
    \item $\widehat{G}$ is the group of 1-dimensional characters of $G$ over $\k$. These are equivalently the group-like elements of $\du{G}$, the Hopf algebra dual to the group algebra $\k G$.
    \item $\Aut_c(G)= C_{\Aut(G)}(\Inn(G))$ is the central automorphism group of $G$.  Equivalently,
    \[ \Aut_c(G) = \{ \phi\in\Aut(G) \ | \ \phi(g)g\inv\in Z(G) \mbox{ for all } g\in G\}.\]
    \item Given another finite group $H$, we call $\phi\in\Hom(G,\widehat{H})$ a $G\times H$ bicharacter.  When $G=H$ we simply say that $\phi$ is a bicharacter (of $G$). We identify $\phi$ with the corresponding bilinear form $G\times H\to U(1)$ given by $(g,h)\mapsto \phi(g)(h)=\phi(g,h)$.
  \end{enumerate}
\end{df}
Unless otherwise noted, $G$ will always denote a group, and all groups will be finite.  Given a group $G$ we fix a (generally not unique) Krull-Schmidt decomposition
\begin{equation}\label{eq:decomp}
    G = G_0\times G_1\times\cdots \times G_n
\end{equation}
where $0\leq n\in \BZ$, $G_0$ is abelian (and possibly trivial), and $G_1,...,G_n$ are indecomposable non-abelian groups.  The value $n$ does not depend on the choice of decomposition, and unless otherwise noted all uses of $n$, such as references to subsets of $\{1,...,n\}$, will refer to this value.

Many of our results will have behavior that depends on whether or not $G_0$ is trivial.  The following classic result is the principle reason for this.
\begin{lem}[{\citep[Theorem 1]{AY65}}]\label{lem:Adney}
    Given a finite group $G$, consider the set map $F\colon \Aut_c(G)\to \Hom(G,Z(G))$ given by $\phi\mapsto \phi-1$.  Then $G$ is purely non-abelian if and only if $F$ is a bijection.  In any case the inverse mapping $\Img(F)\to\Aut_c(G)$ is given by $z\mapsto 1+z$.
\end{lem}

\begin{df}
Given an element $\phi\in\End(G)$ we define the components $\phi_{i,j}\colon G_j\to G_i$ for all $0\leq i,j\leq n$ by the obvious restrictions.
\end{df}
These components naturally define an $(n+1)\times (n+1)$ matrix of morphisms, which are evaluated from the right, and completely determine $\phi$ \citep{BCM}.

\begin{example}
  If $G=G_1\times G_2$ then $\phi\in\End(G)$ is written as
  \[ \begin{pmatrix}
    \phi_{1,1} & \phi_{1,2}\\
    \phi_{2,1} & \phi_{2,2}
  \end{pmatrix},\]
  and its action is
  \[ \begin{pmatrix}
    \phi_{1,1} & \phi_{1,2}\\
    \phi_{2,1} & \phi_{2,2}
  \end{pmatrix}(g,h) = (\phi_{1,1}(g)\phi_{1,2}(h), \phi_{2,1}(g)\phi_{2,2}(h)).\]
  In general the $(i,j)$ entry of $\phi$ as a matrix is precisely $\phi_{i,j}\colon G_j\to G_i$.
\end{example}

\subsection{The morphism \texorpdfstring{$p$}{p}}
A particular type of morphism of Hopf algebras will feature prominently throughout the paper, which we now describe.
\begin{lem}\label{lem:p-props}
Let $p\colon\du{G}\to\k G$ be a morphism of Hopf algebras.  Then the following hold.
\begin{lemenum}
  \item \citep[Theorem 3.1]{K14} There exist isomorphic abelian subgroups $A,B\subseteq G$ such that $\Img(p)=\k B$ and $p(e_g)\neq 0$ if and only if $g\in A$.\label{lem-part:AB-1}
  \item \citep[Theorem 3.1]{K14} $p$ restricts to a group isomorphism $\widehat{A}\to B$.  As a consequence, $\{p(e_a)\}_{a\in A}$ is a basis of $\k B$.\label{lem-part:AB-2}
  \item \citep[Theorem 3.15]{K17:Twisted} There exists a unique $A\times B$ bicharacter $\sigma$ such that for all $a\in A$
  \[ p(e_a) = \frac{1}{|A|}\sum_{b\in B}\sigma(a,b)b.\]
  Moreover, every element of $\widehat{A}$ is equal to $\sigma(\cdot,b)$ for some unique $b\in B$, and every element of $\widehat{B}$ is equal to $\sigma(a,\cdot)$ for some unique $a\in A$.\label{lem-part:sigma}
  \item The set of all such $p$ with $A,B\subseteq Z(G)$ forms a group under the convolution product. This group is canonically isomorphic to $\Hom(\widehat{Z(G)},Z(G))$, with the isomorphism given by the obvious restriction maps.
\end{lemenum}
\end{lem}
\begin{proof}
  The first, second, and third items are given by the indicated references.  The fourth is a consequence of the second.  See also \citep[Proposition 5.2]{K14} and its proof.
\end{proof}
For the remainder of the paper, all uses of $A,B,\sigma$ will refer to the subgroups and bicharacter from the preceding lemma.

\subsection{The double of \texorpdfstring{$G$}{G}}
Next we define the Hopf algebra $\D(G)$, the Drinfeld double of $G$ over $\k$.  As a coalgebra this is $\duc{G}\ot\k G$.  Denoting elements of $\D(G)$ by $f\# g$, $f\in\duc{G}$, $g\in G$, the algebra structure is given by the semidirect product formula
\[ (f\# g)\cdot (f'\# g') = f (g\rhup f')\# gg'.\]
The identity is $\varepsilon\#1$.  The antipode is
\[ S(e_g\# x) = e_{x\inv g\inv x}\# x\inv.\]
$\D(G)$ and $\k G$ are simultaneously semisimple, and in particular are semisimple for any $G$ when $\k=\BC$.

$\D(G)$ admits the structure of a ribbon Hopf algebra in the following two standard ways, the first of which will be of principle interest here.
\begin{df}\label{df:standard}
  The standard quasitriangular structure of $\D(G)$ is
  \[ R_0 = \sum_{g\in G}\varepsilon\# g\ot e_g\# 1,\]
  with Drinfeld and ribbon element
  \[ u_{R_0} = \sum_{g\in G}e_g\# g\inv.\]

  Depending on choices of notation, sometimes the quasitriangular structure of $\D(G)$ is instead defined as
  \[ R_1 = \tau(R_0\inv) = \sum_{g\in G} e_g\# 1\ot \varepsilon\# g\inv,\]
  with Drinfeld and ribbon element
  \[ u_{R_1} = \sum_{g\in G}e_g\# g.\]
\end{df}

However, in general $\D(G)$ can have many additional quasitriangular structures, and the author has completely classified them for arbitrary $G$ in \citep{K3:QTS}. There is also a more categorical classification due to \citet{Nik:braidings}. While the two classifications are necessarily equivalent, there is currently no direct proof of this equivalence.  One of our goals for the paper will be to prove one direction of such an equivalence, by showing the braidings of \citep{K3:QTS} are at least a subset of the braidings in \citep{Nik:braidings}.

We state the desired classification for a group with given decomposition as in \cref{eq:decomp}.

\begin{thm}\citep[Theorem 7.4]{K3:QTS}\label{thm:qts}
	Let $G$ be a group. The quasitriangular structures of $\D(G)$ are those $R\in \D(G)\ot\D(G)$ which can be written in the form
    \begin{align}\label{eq:R-def}
        R = \sum_{s,t,a\in G} e_s\# au(t)\ot r(s)e_t\# p(e_a)v(s\inv)
    \end{align}
    where
	\begin{enumerate}
		\item $p\in\Hom(\widehat{Z(G)},Z(G))$;
		\item $r\in\Hom(G,\widehat{G})$ is a bicharacter;
        \item $u,v\in\End(G)$;
        \item For each $1\leq i\leq n$ exactly one of the following holds:
		\begin{enumerate}
			\item $v_{i,i}\in\Hom(G_i,Z(G_i))$, $u_{i,i}\in\Aut_c(G_i)$;
			\item $v_{i,i}\in\Aut_c(G_i)$, $u_{i,i}\in\Hom(G_i,Z(G_i))$.
		\end{enumerate}
		\item $u,v$ are normal endomorphisms of $G$, meaning that $v(g^x)=v(g)^x$ for all $g,x\in G$, and similarly for $u$.  As a consequence, $\Img(u_{i,j}),\Img(v_{i,j})\subseteq Z(G_j)$ whenever $i\neq j$;
	\end{enumerate}
    We call the quadruple $\cmorph$ the components of $R$, and we equate $R=\cmorph$ whenever convenient.
\end{thm}
Note that $u$ and $v$ have images which commute with each other elementwise, and that there are no restrictions on $v_{0,0}$ and $u_{0,0}$.  For the full details on normal endomorphisms that are relevant to this classification see \citep[Section 2]{K3:QTS}.

\begin{example}
  If $p,r$ satisfy the conditions of the theorem then $\cmorph[1][r][p][0]$, $\cmorph[1][r][0][0]$, $\cmorph[1][0][p][0]$, $\cmorph[0][r][p][1]$, etc. are all quasitriangular structures of $\D(G)$.
\end{example}

\begin{thm}\citep[Theorem 8.1]{K3:QTS}\label{thm:ribbon}
    The Drinfeld element of $(\D(G),R)$ is
    \begin{align}\label{eq:rib-def}
        u_R=\sum_{a,s\in G} r(s)e_{s\inv}\# p(e_a)a\inv v(s\inv)u(s).
    \end{align}
    Moreover, $u_R$ is also a ribbon element for $(\D(G),R)$.
\end{thm}

\subsection{Representations of \texorpdfstring{$\D(G)$}{D(G)}}
Next we recall a few of the basics about $\Rep(\D(G))$, which can be found in \citep{DPR}.  We adopt here the notation used in \citep[Section 15]{K17:Twisted}

Suppose we are given a group $G$, $s\in G$, and a representation $\rho$ of $C_G(s)$ on a finite dimensional $\k$-vector space $V$.  Let $\class(s)=\{s_0=s,s_1,..,s_m\}$, denoted simply $\{s_j\}$, be the conjugacy class of $s$ in $G$.  For each $j$ we let $t_j\in G$ be such that $t_j s t_j\inv=s_j$. This means that the $t_j$ are a complete set of left coset representatives of $C_G(s)$ in $G$. By convention we always select $t_0=1$.  The data $(\{s_j\},\{t_j\},\rho)$ gives a representation $W$ of $\D(G)$, which we identify with this data, in the following manner.  As a vector space $W$ is the graded vector space $\bigoplus_j t_j\ot V$.  We denote a homogeneous element of $W$ by $\ket{t_j,w}$, where $w\in V$.  The left action of $\D(G)$ is then given by
\begin{equation}
  e_g\#x\cdot\ket{t_j,w} = \delta_{g^x,s_k}\ket{t_k,\rho(h)(w)},
\end{equation}
where $xt_j=t_k h$ for some $h\in C_G(s)$.  By definitions, the values of $t_k$ and $h$ are uniquely determined by the given $x,t_j$ (and conversely).  The isomorphism class of $W$ in fact depends only on $\class(s)$ and the isomorphism class of $\rho$.  The irreducible representations of $\D(G)$ are, in particular, parameterized by pairs $(s,\chi)$ where $s\in G$ and $\chi$ is an irreducible character of $C_G(s)$.

\begin{df}\label{def:DG-basis}
  When $\D(G)$ has precisely $m$ irreducible representations, up to isomorphism, for $1\leq a\leq m$ we let $X_a=(t_a,\psi_a)$ denote a fixed representative for an isomorphism class. We then define $\Gamma=\{1,...,m\}$, the set of labels for the given (but usually arbitrary) complete set of representatives for the isomorphism classes.
\end{df}

\begin{rem}
  A more categorical interpretation of $\Rep(\D(G))$ is as $\Z(\VecG_G)$, the categorical center of the pointed fusion category of finite dimensional $G$-graded vector spaces.  As our principle point of attack on this category will be through the Hopf algebra $\D(G)$, we prefer the above characterization.
\end{rem}

\begin{df}\label{def:lambda}
For a group $H$ and a character (or representation) of $H$, we let $\lambda_\chi$ denote the linear character (equiv. representation) of $Z(H)$ obtained by restriction and rescaling: $\lambda_\chi(z)\chi(1) = \chi(z)$ for all $z\in Z(H)$.
\end{df}
By definition $\lambda_\chi$ restricts to an element of $\widehat{C}$ for every subgroup $C\subseteq Z(H)$.

We adopt the convention to omit obvious restriction morphisms whenever convenient.  In particular we omit restriction to the subgroup $C$ as above whenever $C$ is clear from the context, as well as the restriction morphism $\Res_{Z(G)}^G\colon\widehat{G}\to\widehat{Z(G)}$.  This latter case is useful for composing morphisms $p\in\Hom(\widehat{Z(G)},Z(G))$ and $r\in\BCh{G}$, which gives $pr\in\Hom(G,Z(G))$.

\subsection{Modular categories}
Our references for the theory of (modular, braided) tensor and fusion categories will be \citep{BakKir:book,EGNO,mueger:modular}.

\begin{df}
Let $\mathcal{C}$ be a braided fusion category.  We call a complete set of representatives for the isomorphism classes of the simple (=irreducible) objects of $\mathcal{C}$ a basis for $\mathcal{C}$.

We let $\Gamma=\{1,...,m\}$ be a set of labels for an enumeration of a given basis.
\end{df}
\begin{example}
The $X_a=(t_a,\psi_a)$ from \cref{def:DG-basis} define a basis of $\Rep(\D(G))$.
\end{example}

Any braided fusion category defines a pair of matrices, called the $T$-matrix and $S$-matrix (the latter is not to be confused with an antipode), which act on the linear span of the simple objects---equivalently here, they act on the complexification of the Grothendieck ring. Together these are called the modular data, and afford a projective representation of $SL_2(\BZ)$.

We first consider the $T$-matrix.
\begin{df}\label{def:T-def}
Given a braided fusion category $\C$, the $T$-matrix is the diagonal matrix such that for a simple object $X$ then $T_X=T_{X,X}\in U(1)$ describes the pivotal trace of the following diagram on $X$.
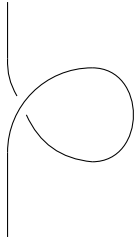
\begin{center}
\begin{tikzpicture}
	\begin{pgfonlayer}{nodelayer}
		\node [style=none] (0) at (-0.25, 2.75) {};
		\node [style=none] (1) at (-0.25, -0.375) {};
		\node [style=none] (2) at (-0.25, .75) {};
		\node [style=none] (3) at (.875, 1.875) {};
		\node [style=none] (4) at (.875, .625) {};
		\node [style=none] (5) at (-0.25, 2) {};
		\node [style=none] (6) at (0, 1.25) {};
		\node [style=none] (7) at (-0.125, 1.5) {};
	\end{pgfonlayer}
	\begin{pgfonlayer}{edgelayer}
		\draw (1.center) to (2.center);
		\draw [in=180, out=90] (2.center) to (3.center);
		\draw [bend right=90, looseness=1.50] (4.center) to (3.center);
		\draw (0.center) to (5.center);
		\draw [bend right=15] (5.center) to (7.center);
		\draw [bend left] (4.center) to (6.center);
	\end{pgfonlayer}
\end{tikzpicture}
\end{center}
    Our string diagrams will be read from top to bottom, and the crossing convention is such that the crossing in the above is the inverse braiding.
\end{df}
For a semisimple ribbon Hopf algebra $(H,R,\nu)$ and simple object $X$, $T_X$ is precisely the scalar by which $\nu\inv$ acts on $X$.

In a slight abuse of notation we let $T^R$ denote the $T$-matrix of $\Rep(\D(G),R,u_R)$.  When $R=R_0$ we simply use $T$.

\begin{example}\label{ex:standard-ribbon}
  For $(\D(G),R_0,u_{R_0})$, the element $u_{R_0}\inv$ acts on $(g,\chi)$ by the scalar $\chi(g)/\chi(1)$.

  Similarly, for $(\D(G),R_1,u_{R_1})$ the element $u_{R_1}\inv$ acts on $(g,\chi)$ by the scalar $\chi(g\inv)/\chi(1)$.  This is simply the conjugate of the scalar we get with the inverse braiding $R_0$.
\end{example}

Next we describe the $S$-matrix.
\begin{df}\label{def:S-matrix}
    Given a braided fusion category $\C$ with braiding $c$ and basis $\{X_1,...,X_m\}$ the $S$-matrix $S$ is the $m\times m$ matrix with $S_{i,j}=S_{X_i,X_j}$ given by the trace of the double-braiding
    \[ c_{X_j,X_i}\circ c_{X_i,X_j} \colon X_i\ot X_j\to X_i\ot X_j.\]
    Graphically, the entries are given by the pivotal trace of the following diagram on pairs of simple objects:
    
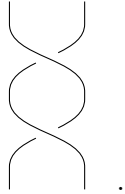
\begin{center}
    \begin{tikzpicture}
      \braid[number of strands=2] (braid) a_1 a_1;
    \end{tikzpicture} \ \ .
    \end{center}
    We say that $\C$ is modular if the $S$-matrix is invertible.
\end{df}
In the case of a quasitriangular Hopf algebra $(H,R)$, the braiding of $\Rep(H,R)$ is given by $v\ot w\to R\inv\cdot w\ot v$, so that the double-braiding is (left) multiplication by $R\inv\tau(R\inv)$. For any quasitriangular structure $R$ of $\D(G)$ we let $S^R$ denote the corresponding $S$-matrix.  When $R=R_0$ we let $S$ stand for $S^{R_0}$.

It is well-known that a symmetric spherical category has an $S$-matrix of rank 1.  Examples of such categories are given by $\Rep(A)$ for any abelian finite group $A$, equipped with the trivial braiding and ribbon element.  Since $\D(A)$ also admits the trivial quasitriangular structure and ribbon element when $A$ is abelian, we can therefore naturally expect that abelian direct factors will introduce pathological behaviors in the modular data for $\Rep(\D(G),R,u_R)$. Namely, the $S$-matrix can fail to be invertible, but the following shows this does not happen for purely non-abelian groups.

\begin{thm}\label{thm:is-modular}
  The following are equivalent for a finite group $G$.
  \begin{enumerate}
    \item $G$ is purely non-abelian.
    \item $\Rep(\D(G),R,u_R)$ is a factorizable modular category for every quasitriangular structure $R$ of $\D(G)$.
  \end{enumerate}
\end{thm}
\begin{proof}
  Every abelian group admits the trivial quasitriangular structure $1\ot 1$, which yields a symmetric category. From this it follows that if $G$ is not purely non-abelian, then there exists a quasitriangular structure which does not yield a modular category.  For the other direction, if $G$ is purely non-abelian then by \citep[Corollary 8.6]{K3:QTS} every quasitriangular structure is factorizable.  Since $\D(G)$ is also semisimple, the desired equivalence then follows from a result of \citet{Takeuchi:ModHopf}, which says that $\Rep(H)$ is modular whenever $H$ is semisimple and factorizable.
\end{proof}

\section{Braid Gauging}\label{sec:gauge}
Given a (braided) fusion category $\C$, there is a universal grading group $\U(\C)$ \citep[Section 3.2]{GN:Nilpotent} and a corresponding degree map $\deg$ on objects.  When the category is modular the universal grading group is abelian.

\begin{example}
  For $G$ a finite group and $\C=\Rep(\D(G))$ we have $\U(C) = G/G'\times \widehat{Z(G)}$, where $\deg(g,\chi) = (gG', \lambda_\chi)$.
\end{example}

\begin{df}[{\citep[Section 4.3]{Nik:braidings}}]\label{def:gauge}
A braid gauging of the braided fusion category $\C$ is a bilinear form $b\colon \U(\C)\times\U(\C)\to\k^\times$.  Given any braiding $c$ of $\C$, we obtain a new braiding $c^b$ by defining
\[ c^b_{X,Y} = b(\deg(X),\deg(Y)) c_{X,Y}\]
for all pairs of simple objects $X,Y$.
\end{df}
Observe that $(c^b_{X,Y})\inv = b(\deg(Y),\deg(X))\inv c_{X,Y}\inv$.  The braid gaugings form a group under convolution products, which is canonically isomorphic to $\BCh{\U(\C)}$.

We adopt the convention to write $b(\deg(X),\deg(Y))$ as simply $b(X,Y)$ whenever $X,Y$ are homogeneous objects under the grading.  Since the tensor product of homogenous objects is again homogeneous, we may also write $b(X,Y)b(X,Z)=b(X,Y\ot Z)$ and $b(X,Z)b(Y,Z)=b(X\ot Y,Z)$.  Of necessity, we additionally have $\deg(X^*)=\deg(X)\inv$ for any homogenous object $X$.

\begin{df}
  Let $(\C,c)$ be a modular fusion category with braiding $c$.

  We say a braid gauging $b$ is modular if $(\C,c^b)$ is modular.

  We say a subgroup $H$ of braid gaugings of $\C$ is modular if $(\C,c^b)$ is modular for all $b\in H$.
\end{df}

Our next goal is to realize braid gaugings in a Hopf algebraic context, in order to express a number of results from \citep{K3:QTS} in categorical language. We first recall the following definition, which is a special case of one appearing in \citep{KS14}.
\begin{df}
  Given Hopf algebras $H,K$, define $\Hom_{Z,coZ}(H,K)$ to consist of those Hopf algebra morphisms $f\colon H\to K$ such that
  \begin{enumerate}
    \item $\Img(f)\subseteq Z(K)$, the center of $K$ as an algebra;
    \item $f(h\com1)\ot h\com 2 = f(h\com 2)\ot h\com 1$ for all $h\in H$.
  \end{enumerate}
  Such a morphism is said to be bicentral.  The image of any such morphism is necessarily both commutative and cocommutative.

  Moreover, the convolution product gives $\Hom_{Z,coZ}(H,K)$ the structure of an abelian group.
\end{df}
We can then translate the categorical notion of braid gaugings into the theory of quasitriangular Hopf algebras as follows.
\begin{thm}\label{thm:bicental-as-gaugings}
  Given a semisimple quasitriangular Hopf algebra $H$, the braid gaugings of $\C=\Rep(H)$ form a group isomorphic to \[\Hom_{Z,coZ}(H^{*\text{co}},H) = \Hom_{Z,coZ}(H^*,H).\]
\end{thm}
\begin{proof}
 Let notation and assumptions be as in the statement.

 Consider $\Hom_{Z,coZ}(H^{*\text{co}},H)$.  By definitions we have the canonical identity $\Hom_{Z,coZ}(H^{*\text{co}},H)=\Hom_{Z,coZ}(H^{*},H)$.  By \citep{ChirKas:HopfCenter} there is a unique maximal Hopf subalgebra $\Z(H)$ in $Z(H)$.  Also by \citep{ChirKas:HopfCenter} there is a unique (up to isomorphism) maximal cocommutative Hopf algebra quotient $CoZ(H^*)$ of $H^*$.  Of necessity $CoZ(H^*)\cong \Z(H)^*$.  Therefore there is a group isomorphism
 \[ \Hom_{Z,coZ}(H^{*\text{co}},H) \cong \Hom(\Z(H)^*,\Z(H)).\]
 By \citep[Theorem 3.8]{GN:Nilpotent} $\Z(H)\cong \du{\U(\C)}$, from which the result now follows.
\end{proof}

\begin{example}\label{ex:weak}
  The Hopf algebra morphisms $\Hom(H^{*\text{co}},H)$ are known as weak $R$-matrices (of $H$) \citep{Jiao:QTSwSP}, and every quasitriangular structure of $H$ can be expressed as a weak $R$-matrix (but not conversely in general) \citep{Rad:QSHA}. This was the basis for the approach taken in \citep{K3:QTS} to classify the quasitriangular structures of $\D(G)$.

  The braid gaugings of $\Rep(H)$ are therefore precisely the bicentral weak $R$-matrices of $H$ (up to isomorphism).
\end{example}

We can characterize the bicentral weak $R$-matrices as follows.
\begin{thm}\label{thm:bicentral}
  Let $H$ be a finite dimensional Hopf algebra, and $\phi\in\Hom(H^{*\text{co}},H)$ a weak $R$-matrix.  Then $\phi$ is bicentral if and only if $\phi$, as an element of $H\ot H$, is in the center of $H\ot H$.
\end{thm}
\begin{proof}
  We let $F\colon H\ot H\to \Hom_{\k}(H^*,H)$ be the injective linear map defined by $F(\sum X^{(1)}\ot X^{(2)})(f) = \sum f(X^{(1)}) X^{(2)}$ as in \citep{Rad:QSHA}. As observed in \citep{Rad:QSHA}, when $H$ is finite dimensional then $\Hom(H^{*\text{co}},H)$ is contained in the image of $F$.  So given $\phi\in\Hom(H^{*\text{co}},H)$, we let $R\in H\ot H$ we such that $F(R)=\phi$.  Given any basis $\mathcal{B}$ of $H$ we can write $R=\sum_{b\in\mathcal{B}} b\ot \phi(b^*)$, where $b^*\in H^*$ is the element dual to $b$.

  Suppose first that $\phi=F(R)$ is bicentral.  Then for any $h\ot k\in H\ot H$ we have
  \[ F(R)*F(h\ot k) = F(R(h\ot k)) = F((h\ot k)R) = F(h\ot k)*F(R).\]
  Since $F$ is injective this implies $R(h\ot k)=(h\ot k)R$.  Since $h\ot k\in H\ot H$ was arbitrary, we conclude that $R$ is in the center of $H\ot H$.

  Conversely, suppose that $R\in Z(H\ot H)$.  For any $f\in H^*$ and $k\in H$, from $R(1\ot k)=(1\ot k)R$ we have
  \[ (f\ot \id)(R(1\ot k)) = \phi(f)k = k\phi(f) = (f\ot \id)((k\ot 1)R).\]
  Similarly, from $R(k\ot 1)=(k\ot 1)R$ we have
  \begin{align*}
    (f\ot\id)(R(k\ot 1)) = f\com2(k)\phi(f\com1) = f\com1(k)\phi(f\com2).
  \end{align*}
  That this is true for all $k\in K$ is equivalent to
  \[ \phi(f\com1)\ot f\com2 = \phi(f\com2)\ot f\com1.\]
  Since $f$ was arbitrary, $\phi$ is bicentral by definition.
\end{proof}

We recall the following classification of certain weak $R$-matrices of $\D(G)$.
\begin{thm}[{\citep[Theorem 7.2]{K3:QTS}}]\label{thm:cwr}
    The weak $R$-matrices of $\D(G)$ that commute with the image of the comultiplication of $\D(G)$ are in bijective correspondence with the quadruples $(w,r,p,z)$ given by:
    \begin{enumerate}
      \item $w,z\in\Hom(G,Z(G))$;
      \item $r\in\BCh{G}$ is a bicharacter;
      \item $p\in\Hom(\widehat{Z(G)},Z(G))$.
    \end{enumerate}
    The morphism $\D(G)^{*\text{co}}\to\D(G)$ is given by
    \begin{equation}\label{eq:cwr-morph}
      e_g\ot x \mapsto \sum_{t\in G} w^*(e_{gt\inv})r(x)\# p(e_t) z(x\inv).
    \end{equation}
    The corresponding element of $\D(G)\ot\D(G)$ may be written as
    \begin{equation}\label{eq:cwr}
        (w,r,p,z)=\sum_{s,a,b\in G} e_s\# a w(b) \ot r(s)e_b\# p(e_a) z(s\inv).
    \end{equation}
\end{thm}
We identify such quadruples with both the morphism and element of $\D(G)\ot\D(G)$ whenever convenient.

\begin{cor}\label{cor:comm-to-bi}
  A weak $R$-matrix of $\D(G)$ is bicentral if and only if, as an element of $\D(G)\ot\D(G)$, it commutes with the image of the comultiplication.
\end{cor}
\begin{proof}
  This follows easily from \cref{thm:cwr,eq:cwr-morph,thm:bicentral}.
\end{proof}

%
%
\section{Braid gaugings and link invariants}\label{sec:links}
Every framed, oriented link can be written as the closure of a string diagram, which in turn allows us to define link invariants in a braided fusion category $\C$ by taking pivotal traces of the corresponding string diagram.  If we use a braid gauging $b$ to change the braiding $c$ as in \cref{def:gauge}, it immediately follows that a link invariant associated to the new braiding is obtained from the old invariant by a multiplicative factor determined by the braid gauging and the simple objects in question.  This is a simple process which we will make clear by demonstrating with a few examples.  Applications of such identities arise when multiple invariants are related to each other, but in general transform under braid gaugings in distinct ways. This can then be used to establish a number of dependency relations among the invariants.  This will be detailed with examples in \cref{sec:verlinde,sec:indicators}.

\subsection{The modular data}\label{sec:modular}
\begin{thm}\label{thm:mod-gauge}
    Let $\C$ be a braided fusion category with braiding $c$.  Fix a braid gauging $b$.  Let $T,S$ be the modular data obtained from $c$, and let $T^b,S^b$ be the modular data obtained from $c^b$.  Then we have
    \begin{gather}
      T^b_X = b(X,X)\inv T_X,\\
      S^b_{X,Y} = b(X,Y) b(Y,X) S_{X,Y}
    \end{gather}
    for all simple objects $X,Y$.
\end{thm}
\begin{proof}
  The diagram for $T_X$ has a single inverse braiding $c_{X,X}\inv$, and since $(c^b_{X,X})\inv = b(X,X)\inv c_{X,X}\inv$ the desired formula for $T^b$ follows.

  Similarly, the diagram for $S_{X,Y}$ has two braidings, $c_{X,Y}$ and $c_{Y,X}$.  Since $c^b_{X,Y} = b(X,Y)c_{X,Y}$ by definition, the desired formula for $S^b$ follows.
\end{proof}
The modularity property is not, in general, invariant under braid gauging. Therefore applying a braid gauging may result in an inequivalent braided tensor category.  Indeed, when $G$ is abelian every quasitriangular structure of $\D(G)$ is also a bicentral weak $R$-matrix, and so every braiding can be gauged to the trivial (symmetric) braiding, or alternatively be gauged to the standard braiding given by $R_0$. \Cref{thm:is-modular} shows that every braiding of $\Rep(\D(G))$ is modular when $G$ is purely non-abelian, however.  In \cref{cor:one-structure} we will show that when $G$ is purely non-abelian then not only does braid gauging leave the modular category $\Rep(\D(G))$ invariant, but that the pivotal category $\Rep(\D(G))$ admits a unique structure as a braided tensor category.

\subsection{The Borromean Rings}
The Borromean rings correspond to the closure of the following braid.
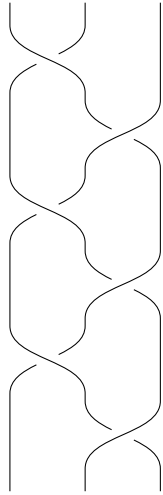
\begin{center}
\begin{tikzpicture}
  \braid[number of strands=3] (braid) a_1 a_2^{-1} a_1 a_2^{-1} a_1 a_2^{-1};
\end{tikzpicture}
\end{center}
In a braided fusion category $\C$ we let $B_{X,Y,Z}$ be the trace of this diagram with the strings labeled left-to-right by the simple objects $X,Y,Z$.  Collectively we refer to this as the $B$-tensor, denoted by $B$. \Citet{KMS:invariant} used the $B$-tensor to show that modular categories are not uniquely determined by their modular data, a fact which was first established in \citep{PM}.

\begin{thm}\label{thm:B-transform}
  Let $\C$ be a braided fusion category with braiding $c$, and let $b$ be a braid gauging.  We let $B^b$ denote the $B$-tensor obtained from the braiding $c^b$ as in \cref{def:gauge}.  Then for all simple objects $X,Y,Z$ we have
  \[ B^b_{X,Y,Z}=B_{X,Y,Z}.\]
  Subsequently, the $B$-tensor is fully invariant under braid gauging.
\end{thm}
\begin{proof}
  The following are the braidings appearing in the diagram for $B_{X,Y,Z}$:
  \begin{gather*}
    c_{X,Y},\
    c_{X,Z}\inv,\
    c_{Y,Z},\\
    c_{Y,X}\inv,\
    c_{Z,X},\
    c_{Y,Z}\inv.
  \end{gather*}
  From \cref{def:gauge} and the following remarks concerning the inverse braidings, it follows that
  \begin{equation*}
    B^b_{X,Y,Z} = \frac{b(X,Y)b(Y,Z)b(Z,X)}{b(Z,X)b(X,Y)b(Z,Y)} B_{X,Y,Z} = B_{X,Y,Z},
  \end{equation*}
  as desired.
\end{proof}

This stands out as an apparent non-example, in the sense that braid gaugings provide no useful information for the invariants.  However, there is still something that can in principle be gleaned: we are free to pick any braid gauging which makes the calculations particularly easy to perform, and we can even pick different braid gaugings for distinct entries if so desired.

\begin{example}
For $\Rep(\D(A))$ with $A$ an abelian group, we may use the symmetric braiding instead of $R_0$, so that \[B_{X,Y,Z}= \dim(X)\dim(Y)\dim(Z)=1\] for all simple objects $X,Y,Z$, which are all 1-dimensional.
\end{example}

\subsection{The Whitehead Link}
The Whitehead link corresponds to the closure of the following link, where of necessity we must label the first and second strands by the same object.
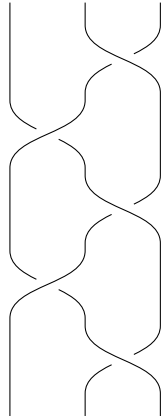
\begin{center}
\begin{tikzpicture}
  \braid[number of strands=3] (braid) a_2 a_1^{-1} a_2 a_1^{-1} a_2;
\end{tikzpicture}
\end{center}
In a braided fusion category $\C$ and simple objects $A,B$ we define $\widetilde{W}_{A,X}$ be the pivotal trace of the above diagram with $A$ on the first two strands and $X$ on the third.  We then define the $W$-matrix $W$ by \[W_{A,X} = T_X \widetilde{W}_{A,X},\] where $T_X$ is the entry corresponding to $X$ in the $T$-matrix.  The Whitehead link is not symmetric as an oriented framed link, so $\widetilde{W}$ will in general not be symmetric.  The normalization has the benefit of making $W$ symmetric \citep[Proposition 2.2]{BDGRTW:BeyondModular}.  The $W$-matrix has also been used to show that modular categories are not uniquely determined by their modular data in \citep[Theorem 4.1]{BDGRTW:BeyondModular}.

\begin{thm}\label{thm:W-transform}
  Let $\C$ be a braided fusion category with braiding $c$, and let $b$ be a braid gauging.  We let $W^b$ denote the $W$-matrix obtained from the braiding $c^b$ as in \cref{def:gauge}.  Then for all simple objects $A,X$ we have
  \[ W^b_{A,X}=b(A,A)b(X,X)W_{A,X}.\]
\end{thm}
\begin{proof}
  The braidings appearing in the braid diagram are
  \begin{gather*}
    c_{A,X}, \ c_{A,X}\inv, \ c_{A,A},\\
    c_{X,A}\inv, \ c_{X,A}.
  \end{gather*}
  By \cref{def:gauge} and the following remarks on inverse braidings we conclude that
  \[ \widetilde{W}^b_{A,X} = \frac{b(A,X)b(A,A)b(X,A)}{b(X,A)b(A,X)} \widetilde{W}_{A,X} = b(A,A) \widetilde{W}_{A,X}.\]
  The desired identity then follows from the definition of $W_{A,X}$.
\end{proof}

The next two sections of the paper turn to working out how braid gauging affects the category $\Rep(\D(G))$, in order to provide an in-depth example.

%
%
\section{Chirality for \texorpdfstring{$\Rep(\D(G))$}{Rep(D(G))}}\label{sec:chirality}
The description of quasitriangular structures in \cref{thm:qts} demonstrates a binary split in the diagonal entries of $u,v$.  This reflects the choice between the standard braiding or its inverse that can be made for each direct factor.  In this section we introduce some definitions and results designed to make it easier to discuss this feature.  This constitutes an expansion of \citep[Definition 11.1]{K3:QTS}.

\begin{df}\label{def:hand}
  We say that a quasitriangular structure $R=\cmorph$ of $\D(G)$ is left-handed if for all $1\leq i\leq n$ we have $u_{i,i}\in\Aut_c(G_i)$ and $v_{i,i}\in\Hom(G_i,Z(G_i))$.  We say that $R$ is right-handed if $\tau(R\inv)$ is left-handed; equivalently, for all $1\leq i\leq n$ we have $v_{i,i}\in\Aut_c(G_i)$ and $u_{i,i}\in\Hom(G_i,Z(G_i))$.
\end{df}
\begin{df}\label{def:chiral}
  For any quasitriangular structure $R=\cmorph$ let $E\subseteq \{1,...,n\}$ consist of those $i$ such that $u_{i,i}\in\Aut_c(G_i)$.  We call this the chiral set, or simply chirality, of $R$.

  We have subgroups $G_E = G_0\times\prod_{i\in E} G_i, G_{E^c} = \prod_{i\not\in E} G_i$ of $G$.  There is a canonical group isomorphism $\mathcal{I}_E\colon G\to G_0\times G_R$ given by the identity on $G_E$ and inversion on $G_{E^c}$.  We identify $\mathcal{I}_E$ as a set map $G\to G$ whenever convenient.

  We let $\pi_E\colon G\to G$ be the canonical retraction to $G_0\times G_E$, and similarly we let $\pi_{E^c}$ be the canonical retraction to $G_{E^c}$.
\end{df}
In the extreme cases, when $E=\{1,...,n\}$ then $\mathcal{I}_E$ is the identity map of $G$, and when $E=\emptyset$ then $\mathcal{I}_E$ is the identity on $G_0$ and the inversion map on $G_1\times\cdots\times G_n$.  These are the cases for left-handed and right-handed quasitriangular structures, respectively.  In all other cases $E$ depends on the choice of decomposition, as we can always permute the orders of the factors.  However, the set of quasitriangular structures is independent of this choice, and once one decomposition is given we can always pick an ordering of the factors in any other decomposition so that the set $E$ remains invariant.

\begin{df}
  Given $E\subseteq\{1,...,n\}$ we define the standard $E$-chiral quasitriangular structure $R_E = \cmorph[\pi_E][0][0][\pi_{E^c}]$.  The Drinfeld and ribbon element of $R_E$ is $u_E = \sum_{g\in G} e_g\# \mathcal{I}_E(g\inv)$.
\end{df}
Note that this definition again depends on the choice of decomposition whenever $E$ is a proper non-empty subset, but we will only ever work in a single, but otherwise arbitrary, decomposition at a time.   There is also a canonical isomorphism of quasitriangular Hopf algebras $(\D(G),R_E)\cong (\D(G_E),R_0)\ot (\D(G_{E^c}),R_1)$.

We observe that we can always partition the braidings of a braided fusion category $\C$ into the orbits under the action of the braid gaugings. The following shows that the standard $E$-chiral quasitriangular structures of $\D(G)$ give representatives for every such orbit.

\begin{prop}\label{prop:qts-gauge}
  Let $R=\cmorph$ be a quasitriangular structure of $\D(G)$.  Then $R$ is $E$-chiral if and only if $R=(z_u,r,p,z_v)R_E$.

  Moreover, given any bicentral weak $R$-matrix $X\in\D(G)\ot\D(G)$, then $XR$ is a quasitriangular structure with the same chirality as $R$.
\end{prop}
\begin{proof}
  Let $\cmorph$ be an $E$-chiral quasitriangular structure of $\D(G)$.  From \cref{thm:qts} and \citep[Theorem 3.14]{K3:QTS} it follows that there exists $z_u,z_v\in\Hom(G,Z(G))$ such that
\begin{align}
  u &= z_u + \pi_E,\label{eq:zu-def}\\
  v &= z_v + \pi_{E^c}.\label{eq:zv-def}
\end{align}
The identity $R=(z_u,r,p,z_v)R_E$ now follows from multiplication in $\D(G)\ot\D(G)$.

On the other hand, if $R=(u',r',p',v')=(w,r,p,z)R_E$ for some bicentral weak $R$-matrix $(w,r,p,z)$, then a direct check shows that $u'= w+\pi_E$, $v'=z+\pi_{E^c}$, $r'=r$, and $p'=p$.  Therefore $R$ is $E$-chiral.

For the second part, given a bicentral weak $R$-matrix $X$, then letting $F$ be defined as in the proof of \cref{thm:bicentral}, the properties of $X$ guarantee that $F(X)*F(R) = F(XR)\in \Hom(\DDG,\D(G))$.  The desired result then follows from \cref{cor:comm-to-bi} and the previous part.

This completes the proof.
\end{proof}

We now desire an explicit formula for the braid gauging corresponding to a given bicentral weak $R$-matrix of $\D(G)$.  We first need the following result.

\begin{lem}\label{lem:sum-to-ele}\label{lem:sum-to-ele-pow}
  Let $G$ be a finite group, and $p\in\Hom(\widehat{Z(G)},Z(G))$.  Suppose $H$ is a subgroup of $G$ containing $A$ and $B$.  Then for any characters $\chi,\eta$ of $H$ we have
  \[ \sum_{a\in A} \lambda_\chi(p(e_a))\lambda_\eta(a\inv) = \lambda_\chi(p(\lambda_\eta)\inv) = \lambda_\eta(p^*(\lambda_\chi)\inv).\]

  Moreover, with $\sigma$ the $A\times B$ bicharacter determined by $p$, then for all $a\in A$ and $b\in B$ we have
  \begin{equation}\label{eq:xp-vals}
  \begin{split}
    \ip{\lambda_\eta}{\sigma(\cdot,b)}_A = \delta_{b\inv,p(\lambda_\eta)},\\
     \ip{\lambda_\chi}{\sigma(a,\cdot)}_B = \delta_{a\inv,p^*(\lambda_\chi)},
  \end{split}
  \end{equation}
  where $\ip{\cdot}{\cdot}_K$ denotes the standard inner product of characters of the group $K$.
\end{lem}
\begin{proof}
  Let assumptions and notation be as in the statement.

  We have
  \begin{align*}
    \sum_{a\in A} \lambda_\chi(p(e_a))\lambda_\eta(a\inv) &= \sum_{b\in B}\ip{\sigma(\cdot,b)}{\lambda_\eta}_A \lambda_\chi(b)\\
    &= \sum_{a\in A} \ip{\lambda_\chi}{\sigma(a,\cdot)}_B \lambda_\eta(a).
  \end{align*}
  By \cref{lem:p-props} there are unique $x\in A,y\in B$ such that $\ip{\lambda_\chi}{\sigma(a,\cdot)}_B=\delta_{a,x}$ and $\ip{\sigma(\cdot,b)}{\lambda_\eta}_A=\delta_{b,y}$ for all $a\in A$ and $b\in B$.  For these $x,y$ we therefore conclude that
  \[ \sum_{a\in A} \lambda_\chi(p(e_a))\lambda_\eta(a\inv) = \lambda_\chi(y) = \lambda_\eta(x).\]

  We next observe that
  \begin{align*}
    p(\lambda_\eta) = \frac{1}{|A|}\sum_{\substack{a\in A\\b\in B}} \sigma(a,b) \lambda_\eta(a)b = \sum_{b\in B} \ip{\lambda_\eta}{\sigma(\cdot,b)}_A b\inv.
  \end{align*}
  Therefore $p(\lambda_\eta)=y\inv$.  A similar calculation shows that $p^*(\lambda_\chi)=x\inv$.

  This completes the proof of all claims.
\end{proof}

\begin{thm}\label{thm:weak-gauge}
  Let $(w,r,p,z)$ be a bicentral weak $R$-matrix of $\D(G)$.  Then the corresponding braid gauging is given by
  \begin{equation}\label{eq:weak-to-gauge}
  \begin{split}
    b((gG',\lambda),(hG',\alpha)) &= r(h\inv,g) \lambda(p(\alpha)\inv z(h))\alpha(w(g\inv))\\
    &= r(h\inv,g) \lambda(z(h))\alpha(p^*(\lambda)\inv w(g\inv)).
  \end{split}
  \end{equation}
\end{thm}
\begin{proof}
  Let $(w,r,p,z)$ be a bicentral weak $R$-matrix of $\D(G)$.  The braiding determined by $R=(w,r,p,z)R_E$ is given by the left action of
  \[R\inv = (-w,-r,-p,-z)R_{E}\inv = R_E\inv(-w,-r,-p,-z).\] So to compute the formula for the desired braid gauging we need only compute the action of $(-w,-r,-p,-z)$.

  Let $(g,\chi),(h,\eta)$ be any two simple objects of $\Rep(\D(G))$, and let $\ket{t_i,m}\in(g,\chi),\ket{t_k,n}\in(h,\eta)$ be any two vectors.  By \cref{eq:cwr} and the fact that $w,r,z$ are all class functions we have
  \begin{align*}
    (-w,-r,-p,-z)\cdot&(\ket{t_k,n}\ot\ket{t_i,m})\\
    &= r(h\inv,g) \sum_{a\in A}\ket{t_k,(a\inv w(g\inv))\cdot n}\\
    &\qquad{}\qquad{} \ot \ket{t_i, (p(e_a)z(h))\cdot m}\\
    &= r(h\inv,g) \lambda_\eta(w(g\inv))\lambda_\eta(z(h))\\
    &\qquad{}\Big(\sum_{a\in A}\lambda_\eta(a\inv)\lambda_\chi(p(e_a))\Big)\ket{t_i,m}\ot \ket{t_k,n}.
  \end{align*}
  The desired formula then follows from \cref{lem:sum-to-ele}.
\end{proof}
We identify the bicentral weak $R$-matrix $(w,r,p,z)$ with the braid gauging given by \cref{eq:weak-to-gauge} for the remainder of the paper.

\begin{rem}\label{rem:mine-is-subset}
  This proves that the braidings obtained from \citep{K3:QTS}, which we use here, are included in the braidings constructed in \citep{Nik:braidings}. The former are all given by braid gaugings of the usual braidings obtained from the $R_E$, which are themselves obtained from standard braidings and their inverses.  These two descriptions are both derived as complete descriptions of all possible braidings, so must yield all of the same braidings.  However the descriptions are rather distinctive, and proceed from distinct starting points.  This makes a translation between them desirable. A direct translation going in the other direction has not yet been determined.
\end{rem}

\begin{df}\label{df:z-def}
  Let $R=\cmorph$ be an $E$-chiral quasitriangular structure of $\D(G)$. By \cref{thm:qts} we may define $z_R\in\Hom(G,Z(G))$ by
  \[  z_R(g)\mathcal{I}_E(g) = u(g)v(g\inv)\]
  for all $g\in G$.

  Equivalently, $z_R=z_u-z_v$, with $z_u,z_v$ defined as in \cref{eq:zu-def,eq:zv-def}.
\end{df}

\begin{example}
  For the standard $E$-chiral quasitriangular structure $R_E$ we have that $z_{R_E}$ is the trivial morphism.
\end{example}
\begin{example}
  The particular case $R=\cmorph[1][r][p][z]$, which will appear frequently in \cref{sec:indicators,sec:verlinde}, has $z_R=-z$.  On the other hand, $R=\cmorph[z][r][p][1]$ has $z_R=z$.
\end{example}

\begin{cor}\label{cor:DG-mod-gauge}
  Let $R=\cmorph$ be an $E$-chiral quasitriangular structure of $\D(G)$.  Then for the simple objects $(g,\chi),(h,\eta)$ of $\Rep(\D(G))$ we have
  \begin{gather*}
    T_X^R = r(g,g)\lambda_\chi(p(\lambda_\chi) z_R(g\inv)) T_X^{R_E},\\
    S_{X,Y}^R = \frac{\lambda_\chi(p(\lambda_\eta)\inv z_R(h\inv)) \lambda_\eta(p(\lambda_\chi)\inv z_R(g\inv))}{r(g,h)r(h,g)} S_{X,Y}^{R_E}.
  \end{gather*}
\end{cor}
\begin{proof}
  Apply \cref{df:z-def,thm:weak-gauge} to \cref{thm:mod-gauge}.
\end{proof}

%
%
\section{Invariance of \texorpdfstring{$\Rep(\D(G))$}{Rep(D(G))} under change of braiding}\label{sec:invariance}
Our goal now is to show that the category $\Rep(\D(G))$ does not change under a change of braiding whenever $G$ is purely non-abelian.  We first need a few technical lemmas.

\begin{lem}\label{lem:finite-comp}
  Let $G$ be a group.  The following are equivalent:
  \begin{enumerate}
    \item $G$ is purely non-abelian.
    \item There exists $N\in\BN$ such that for any sequence $\{f_m\}_{m=1}^\infty$ in $\Hom(G,Z(G))\subseteq \End(G)$, the sequence $\{g_m\}_{m=1}^\infty$ inductively defined by $g_1=f_1$ and $g_{m+1}=f_{m+1}g_m$ satisfies $g_k = 0$ (the trivial morphism) for all $k\geq N$.
  \end{enumerate}
\end{lem}
\begin{proof}
  The reverse direction of the equivalence follows by the contrapositive and Fitting's Lemma, which implies that $G$ has a non-trivial abelian direct factor if and only there is an element of $\Hom(G,Z(G))$ which is not nilpotent.

  So consider the forward direction and sequences $\{f_m\}_{m=1}^\infty$, $\{g_m\}_{m=1}^\infty$ as in the statement.  Since $G$ is finite, so is $\Hom(G,Z(G))$.  Therefore in the set with repetitions $\{g_m\}_{m=1}^{|\Hom(G,Z(G))|}$ either the trivial morphism appears, or some non-trivial element $g\in\Hom(G,Z(G))$ appears at least twice.  If $g_m$ is trivial for some $m$ then so is $g_{m+1}$, so we need only consider the second case.  By definition of the sequence $\{g_m\}_{m=1}^\infty$ this implies there exists $f\in\Hom(G,Z(G))$ such that $fg = g$.  Therefore by induction for all $n\in\BN$ $f^n g = g$. But from Fitting's Lemma we know that if $G$ is purely non-abelian then every element of $\Hom(G,Z(G))$ is nilpotent. Therefore $f^n$ is the trivial morphism for all sufficiently large $n$, and we conclude that $g$ is the trivial morphism.  It follows that by taking $N=|\Hom(G,Z(G))|$ we obtain the desired result.
\end{proof}
Equivalently, the result says that $G$ is purely non-abelian if and only if every composition of $k\geq |\Hom(G,Z(G))|$ elements of $\Hom(G,Z(G))$ is necessarily trivial.  This will be critically important for this section's goal, as in particular we will need the following result.

\begin{lem}\label{lem:recursion}
    Let $G$ be a purely non-abelian group. Then for any $m,t\in\BN$, $w\in\Hom(G,Z(G))$, $f_1,...,f_m\in \Hom(\widehat{Z(G)},Z(G))$, and $g_1,...,g_m\in\BCh{G}$ there is a unique morphism $z\in\Hom(G,Z(G))$ satisfying
    \[ z = w+\sum_{s=1}^t\sum_{i=1}^m f_i (z^*)^s g_i.\]
\end{lem}
\begin{proof}
  Suppose $G$ is purely non-abelian. Define a set map \[f\colon \Hom(G,Z(G))\to\Hom(G,Z(G))\] by $f(z)= w+\sum_{s=1}^t\sum_{i=1}^m f_i (z^*)^s g_i$. We have
  \[ f^2(z)=f(f(z)) = w + \sum_{i_1=1}^m\sum_{s_1=1}^t f_{i_1}\left(w^* + \sum_{i_2=1}^m\sum_{s_2=1}^t g_{i_2}^* z^{s_2} f_{i_2}^*\right)^{s_1} g_i.\]
  We see that all terms involving $z$ are a composition of at least three elements of $\Hom(G,Z(G))$.  Moreover, by induction on $k\in\BN$ we see that in $f^{2k}(z)$ every term involving $z$ is a composition of at least $2k+1$ elements of $\Hom(G,Z(G))$.  By \cref{lem:finite-comp} there exists $N\in\BN$ such that for all $k\geq N$ $f^{2k}(z)$ does not depend on $z$.  So for any $z_0\in\Hom(G,Z(G))$ it follows that $z=f^{2N}(z_0)$ is the unique solution to the desired identity.
\end{proof}

\begin{example}
  The trivial morphism $z=0$ is the unique solution if and only if $w=0$.  At the other extreme, if the $f_i,g_i$ are all trivial then $z=w$ is the unique solution.
\end{example}

\begin{example}
  Let $A=\cyc{a}$ be a non-trivial cyclic group, let $\widehat{A}=\cyc{\alpha}$.  Define $p\colon \widehat{A}\to A$ by $p(\alpha)=a\inv$, and define $r\colon A\to \widehat{A}$ by $r(a)=\alpha$.  Any $w\in\Hom(A,Z(A))=\End(A)$ is completely determined by $w(a)=a^t$ for some $t\in\{0,...,|A|-1\}$.  Fix such $w,t$.  Consider any $z\in\Hom(G,Z(G))$, determined by $z(a)=a^s$.  Then we compute that $(pz^*r)(a) = a^{-s}$, so that $(w+pz^*r)(a) = a^{t-s}$.  Therefore $z=w+pz^*r$ if and only if $2s\equiv t\bmod |A|$.  If $|A|$ is odd then we can always find such a value of $s$.  However, if $|A|$ is even and $t$ is odd then no such $z$ exists.

  If we define $f(z)=w+pz^*r$ as in the proof of \cref{lem:recursion}, then $f^k(w)=0$ if $k$ is odd and $f^k(w)=w$ if $k$ is even.  This is independent of the order of $A$, so in the presence of abelian direct factors we see that solutions may not always exist, and even when they do the iterates of $f$ need not converge.
\end{example}

\begin{lem}\label{lem:pna}
  The following are equivalent for a group $G$:
  \begin{enumerate}
    \item $G$ is purely non-abelian.\label{item:1}
    \item For all $p\in\Hom(\widehat{Z(G)},Z(G))$ and $r\in\BCh{G}$ there exists $\delta\in\Aut_c(G)$ such that\label{item:2}
    \[ \delta = 1 + p(\delta\inv)^* r.\]
    \item For all $p\in\Hom(\widehat{Z(G)},Z(G))$ if we define a group homomorphism $T_p\colon \BCh{G}\to\Hom(G,Z(G))$ by $T_p(\beta)=p\beta$; let $\beta_1,...,\beta_m$ be any complete set of (left) coset representatives of $\ker(T_p)$; and define $\delta_k=1+p\beta_k\in\End(G)$ for all $1\leq k\leq m$ then
        \[ \BCh{G} = \bigcup_{k=1}^m \delta_k^*(\beta_k\ker(T_p)).\]\label{item:3}
    \item For all $r\in\BCh{G}$ if we define a group homomorphism $S_r\colon \Hom(\widehat{Z(G)},Z(G))\to \Hom(G,Z(G))$ by $S_r(\gamma)=\gamma r$; let $\gamma_1,...,\gamma_m$ be any complete set of (right) coset representatives of $\ker(S_r)$; and define $\delta_k = 1+\gamma_k r\in\End(G)$ for all $1\leq k \leq m$ then
        \[ \Hom(\widehat{Z(G)},Z(G)) = \bigcup_{k=1}^m (\ker(S_r) \gamma_k)\delta_k^*.\]\label{item:4}
  \end{enumerate}
\end{lem}
\begin{proof}
  We first show that \cref{item:3} implies \cref{item:1} via the contrapositive.  Suppose that $G=A\times H$ for some non-trivial cyclic group $A=\cyc{a}$.  Let $\widehat{A}=\cyc{\alpha}$. Then there are $r\in\BCh{G},p\in\Hom(\widehat{Z(G)},Z(G))$ defined by $r((a^i,h))=\alpha^i\ot\varepsilon$ for all $h\in H$, and $p(\alpha,\chi)=(a\inv,1)$ for all $\chi\in\widehat{Z(H)}$.  For this $p$ we consider $T_p$.  Without loss of generality we may pick the set of coset representatives to have $\beta_k=r$ for some $k$, so that $r\not\in\ker(T_p)$ and $\delta_k = 1+pr$. By order considerations the union in question is all of $\BCh{G}$ if and only if $\delta_m^*$ is injective on $\beta_m\ker(T_p)$ for all $m$, and the union is a disjoint union. Now by definitions $(1+pr)^* r = r - r = 0$.  Thus $\delta_k^*(\beta_k\ker(T_p))\cap \ker(T_p)\neq \emptyset$, and so the union does not give all of $\BCh{G}$, as desired.

  A similar argument shows that \cref{item:4} implies \cref{item:1}.

  We next show that \cref{item:1} implies \cref{item:2}.  By \cref{lem:Adney} we may write any $\delta\in\Aut_c(G)$ as $\delta=1-z$ for some $z\in\Hom(G,Z(G))$.  Fitting's Lemma implies that any such $z$ is nilpotent.  We may subsequently write $\delta\inv = 1 + \sum_{j=1}^m z^j$ for some $0\leq m\in\BZ$. Then
  \[ \delta = 1+p(\delta\inv)^* r \iff z = -pr + \sum_{j=1}^m -p (z^j)^* r.\]
  An application of \cref{lem:recursion} guarantees a solution $z\in\Hom(G,Z(G))$ exists.  Therefore \cref{item:1} implies \cref{item:2} as desired.

  Suppose next that \cref{item:3} holds.  Then for any $r\in\BCh{G}$ there exists $k$ such that $r\in\delta_k^*(\beta_k \ker(T_p))$.  By the previous case we know that $G$ is purely non-abelian, so that by \cref{lem:Adney} we have $\delta_k\in\Aut_c(G)$.  Moreover,
  \[ p(\delta_k\inv)^* r = p(\delta_k\inv)^*(\delta_k^*(\beta_k+\beta_0)) = p(\beta_k+\beta_0)=p\beta_k\]
  for some $\beta_0\in\ker(T_p)$.  Therefore $\delta_k = 1+ p(\delta_k\inv)^* r$ and \cref{item:2} holds.

  A similar argument shows that \cref{item:4} implies \cref{item:2}.

  We need only show that \cref{item:2} implies \cref{item:3,item:4} to complete the proof.  Fix $p$ and define $T_p,\beta_k,\delta_k$ as in \cref{item:3}.  If \cref{item:2} holds, then given any $r\in\BCh{G}$ and $\delta\in\Aut_c(G)$ satisfying $\delta=1+p(\delta\inv)^*r$ we observe that of necessity $\delta-1=p(\delta\inv)^*r\in\Img(T_p)$, so that in fact $\delta=\delta_k$ for some $k$.  By definitions we must have $r\in \delta_k^*(\beta_k\ker(T_p))$.  Since $r$ was arbitrary, \cref{item:3} holds.  A similar argument shows that \cref{item:2} implies \cref{item:4} as desired, and so completes the proof.
\end{proof}

We will also need the following result, which will allow us to rearrange certain expressions involving multiple characters.
\begin{lem}\label{lem:char-swap}
Let $G$ be a group.  Let $p\in\Hom(\widehat{Z(G)},Z(G))$, $\beta\in\BCh{G}$, and $\lambda,\alpha\in\widehat{Z(G)}$. Then the following hold.
\begin{enumerate}
\item $\lambda(p(\alpha)) = \alpha(p^*(\lambda)).$

\item For all $h\in G$, $\beta(p(\lambda),h) = \lambda((p^*\beta^*)(h))$.
\end{enumerate}
\end{lem}
\begin{proof}
  Both identities follow by simply expanding the definitions in the standard basis of $\du{G}$.
\end{proof}

We can now consider the invariance of $\Rep(\D(G))$ under braid-gauging when $G$ is purely non-abelian.
\begin{thm}\label{thm:chiral-convert}
Let $G$ be a purely non-abelian group, and let $R=\cmorph$ be an $E$-chiral quasitriangular structure of $\D(G)$. Then there exists a braided tensor equivalence
\[\Rep(\D(G),R)\to\Rep(\D(G),R_E).\]
\end{thm}
\begin{proof}
  Let $G,R$ be as in the statement.  We proceed by considering a fairly general candidate for a braided tensor equivalence, and specialize along the way to construct a specific example.  While we could simply start with the specific equivalence, deriving the equation in more generality has the benefit of providing insights into the more general form of such equivalences, as well as the reasons for the particular example we single out.

  Consider \[ \psi=\begin{pmatrix} \alpha&\beta\\
            \gamma & \delta \end{pmatrix} \in \Aut(\D(G)).\]
  By \citep[Theorem 15.3]{K:Twisted} and its proof we have the following irreducible objects and corresponding data in $\Rep(\D(G))$
  \begin{gather*}
    (g,\chi) = (\{g_i\}, \{\delta(s_i)\},\rho),\\
    (h,\eta) = (\{h_k\}, \{\delta(t_k)\},\zeta),\\
    (\alpha^*(g),\delta^*\chi) = (\{\alpha^*(g_i)\},\{s_i\},\rho\circ\delta),\\
    (\alpha^*(h),\delta^*\eta) = (\{\alpha^*(h_k)\},\{t_k\},\zeta\circ\delta),\\
    (\gamma^*(\lambda_\chi),\beta^*(g)) = (\gamma^*(\lambda_\chi),1,\beta^*(g)),\\
    (\gamma^*(\lambda_\eta),\beta^*(h)) = (\gamma^*(\lambda_\eta),1,\beta^*(h)),
  \end{gather*}
  such that if $\Psi$ is the tensor autoequivalence induced by $\psi$ then $\Psi$ is given on irreducibles (up to natural transformations) by
  \[ \Psi(g,\chi) = (\alpha^*(g),\delta^*\chi)\ot (\gamma^*(\lambda_\chi),\beta^*(g))\]
  and on vectors $\ket{\delta(s_i),m}\in(g,\chi)$ by
  \[ \ket{\delta(s_i),m}\mapsto\ket{s_i,m}\ot\ket{1,1}.\]
  The definition on vectors in particular tells us how $\Psi$ acts on morphisms.  To emphasize that the representation underlying $\ket{s_i,m}$ in the above is $\rho\circ\delta$, we will instead write
  \[ \ket{\delta(s_i),m}\mapsto\ket{s_i,m^\delta}\ot\ket{1,1}.\]

  Our goal is to find a choice of $\Psi$ and a tensor structure $J$, as in \citep[Definition 2.4.1]{EGNO}, such that $(\Psi,J)$ becomes a braided tensor autoequivalence $\Rep(\D(G),R)\to \Rep(\D(G),R_E)$, as defined by \citep[Definition 8.1.7]{EGNO}.  We first observe that if $J=(w,\tilde{r},\tilde{p},z)$ is a braid gauging, then in fact $J$ satisfies the axioms to be a tensor structure.  Our claim will be that exactly such a choice of $J$ will suffice.  So we need to compute the horizontal rows in \citep[Equation (8.5)]{EGNO}.

  We first explicitly compute the braiding $(g,\chi)\ot(h,\eta)\to (h,\eta)\ot (g,\chi)$ given by $R$.  For $\ket{\delta(s_i),m}\in(g,\chi)$ and $\ket{\delta(t_k),n}\in(h,\eta)$, by \cref{thm:weak-gauge} this is given by
  \begin{align*}
    &\ket{\delta(s_i),m}\ot\ket{\delta(t_k),n}\mapsto R\inv \ket{\delta(t_k),n}\ot\ket{\delta(s_i),m}\\
    &= \lambda_\chi(p(\lambda_\eta)\inv z_v(h))\lambda_\eta(z_u(g\inv))r(h\inv,g)R_E\inv\cdot(\ket{\delta(t_k),n}\ot\ket{\delta(s_i),m}).
  \end{align*}

  So we need to compute the action of $R_E\inv$, which is given by
  \begin{align*}
    \Big(\sum_{x,y\in G}& e_{x\inv}\#\pi_E(y\inv)\ot e_y\# \pi_{E^c}(x\inv)\Big)\cdot(\ket{\delta(t_k),n}\ot\ket{\delta(s_i),m})\\
    &= \Big(\sum_{{\mathclap{\substack{x,y\in G\\j,l\\\pi_E(y\inv)\delta(t_k) = \delta(t_l)\tilde{y}\\\pi_{E^c}(x\inv)\delta(s_i)=\delta(s_j)\tilde{x}\\x\inv = h_l\\y = g_j}}}} e_{x\inv}\#\pi_E(y\inv)\ot e_y\# \pi_{E^c}(x\inv)\Big)\cdot(\ket{\delta(t_k),n}\ot\ket{\delta(s_i),m})\\
    &= \Big(\sum_{{\mathclap{\substack{j,l\\\pi_E(g_j\inv)\delta(t_k) = \delta(t_l) \tilde{y}\\\pi_{E^c}(h_l)\delta(s_i)=\delta(s_j) \tilde{x}}}}} e_{h_l}\#\pi_E(g_j\inv)\ot e_{g_j}\# \pi_{E^c}(h_l)\Big)\cdot(\ket{t_k,n}\ot\ket{s_i,m})\\
    &= \sum_{\substack{j,l\\\pi_E(g_j\inv)\delta(t_k) = \delta(t_l) \tilde{y}\\\pi_{E^c}(h_l)\delta(s_i)=\delta(s_j) \tilde{x}}} \ket{\delta(t_l),\tilde{y}\cdot n}\ot \ket{\delta(s_j),\tilde{x}\cdot m}.
  \end{align*}
  Applying $\Psi$ we find that the bottom row of \citep[Equation (8.5)]{EGNO} is given by
  \begin{align}\label{eq:bottom-EGNO}
  \begin{split}
     (&\ket{s_i,m^\delta}\ot\ket{1,1})\ot(\ket{t_k,n^\delta}\ot\ket{1,1})\\
     & \mapsto \lambda_\chi(p(\lambda_\eta)\inv z_v(h))\lambda_\eta(z_u(g\inv))r(h\inv,g)\\
     &\qquad{}\sum_{{\mathclap{\substack{j,l\\\pi_E(g_j\inv)\delta(t_k) = \delta(t_l) \tilde{y}\\\pi_{E^c}(h_l)\delta(s_i)=\delta(s_j) \tilde{x}}}}}
        (\ket{t_l,(\tilde{y}\cdot n)^\delta}\ot\ket{1,1})\ot(\ket{s_j,(\tilde{x}\cdot m)^\delta}\ot\ket{1,1}).
  \end{split}
  \end{align}

  To find the top row of \citep[Equation (8.5)]{EGNO} we compute the action of $R_E\inv$ on $\Psi(h,\eta)\ot\Psi(g,\chi)$ as
  \begin{align*}
    \Big(&\sum_{x,y\in G} e_{x\inv}\#\pi_E(y\inv)\ot e_y\# \pi_{E^c}(x\inv)\Big)\\
    &\qquad{}\cdot(\ket{t_k,n^\delta}\ot\ket{1,1})\ot(\ket{s_i,m^\delta}\ot\ket{1,1})\\
    &=\Big(\sum_{{\mathclap{\substack{x,y\in G\\j,l\\\pi_E(y\inv)t_k=t_l y'\\\pi_{E^c}(x\inv)s_i = s_j x'\\x\inv=\alpha^*(h_l) \gamma^*(\lambda_\eta)\\y=\alpha^*(g_j)\gamma^*(\lambda_\chi)}}}}
        e_{x\inv}\#\pi_E(y\inv)\ot e_y\# \pi_{E^c}(x\inv)\Big)(\ket{t_k,n^\delta}\ot\ket{1,1})\ot(\ket{s_i,m^\delta}\ot\ket{1,1})\\
    &=\beta(\pi_E(\alpha^*(g)\gamma^*(\lambda_\chi))\inv,h) \beta(\pi_{E^c}(\alpha^*(h)\gamma^*(\lambda_\eta)),g)\\
    &\qquad{}\Big(\sum_{{\mathclap{\substack{j,l\\\pi_E(\alpha^*(g_j\inv)\gamma^*(\lambda_\chi)\inv) t_k = t_l y'\\\pi_{E^c}(\alpha^*(h_l)\gamma^*(\lambda_\eta))s_i= s_j x'}}}}
        (\ket{t_l,y'\cdot n^\delta}\ot \ket{1,1}) \ot ( \ket{s_j, x'\cdot m^\delta}\ot\ket{1,1})\Big)
  \end{align*}
  Therefore the top row of \citep[Equation (8.5)]{EGNO} is given by
  \begin{align}\label{eq:top-EGNO}
    \begin{split}
        (&\ket{s_i,m^\delta}\ot\ket{1,1})\ot(\ket{t_k,n^\delta}\ot\ket{1,1})\\
        &\mapsto
      \beta(\pi_E(\alpha^*(g)\gamma^*(\lambda_\chi))\inv,h) \beta(\pi_{E^c}(\alpha^*(h)\gamma^*(\lambda_\eta)),g)\\
    &\qquad{}\Big(\sum_{{\mathclap{\substack{j,l\\\pi_E(\alpha^*(g_j\inv)\gamma^*(\lambda_\chi)\inv) t_k = t_l y'\\\pi_{E^c}(\alpha^*(h_l)\gamma^*(\lambda_\eta))s_i= s_j x'}}}}
        (\ket{t_l,y'\cdot n^\delta}\ot \ket{1,1}) \ot ( \ket{s_j, x'\cdot m^\delta}\ot\ket{1,1})\Big).
    \end{split}
  \end{align}

  We now wish to compare the summations in \cref{eq:bottom-EGNO,eq:top-EGNO}, so as to decide exactly which choice of $\psi$ leaves these two expressions differing only by a scalar multiple.  The summation in \cref{eq:top-EGNO} involves the identities
  \begin{align*}
    \pi_E(\alpha^*(g_j\inv)\gamma^*(\lambda_\chi)\inv) t_k &= t_l y',\\
    \pi_{E^c}(\alpha^*(h_l)\gamma^*(\lambda_\eta))s_i &= s_j x'
  \end{align*}
  Applying $\delta$ to these we have
  \begin{align*}
    \delta\left(\pi_E(\alpha^*(g_j\inv)\gamma^*(\lambda_\chi)\inv)\right) \delta(t_k) &= \delta(t_l) \delta(y'),\\
    \delta\left(\pi_{E^c}(\alpha^*(h_l)\gamma^*(\lambda_\eta))\right)\delta(s_i) &= \delta(s_j) \delta(x').
  \end{align*}
  Since $\gamma^*$ has central image and all of $\pi_E,\pi_{E^c},\delta$ send central elements to central elements, these are equivalent to
  \begin{align*}
    \delta(\pi_E(\alpha^*(g_j)))\inv \delta(t_k) &= \delta(t_l)\delta(y'\pi_E(\gamma^*(\lambda_\chi))),\\
    \delta(\pi_{E^c}(\alpha^*(h_l)))\delta(s_i) &= \delta(s_j)\delta(x'\pi_{E^c}(\gamma^*(\lambda_\eta))\inv).
  \end{align*}

  Comparing this with \cref{eq:bottom-EGNO} we see we are concerned with relating $\pi_E\alpha^*$ to $\pi_E$, and $\pi_{E^c}\alpha^*$ to $\pi_{E^c}$.  We recall that $\delta\alpha^*$ always defines an element of $\Aut_c(G)$.  So consider the case where in fact $\alpha^*\in\Aut_c(G)$, so that $\alpha^*=1+z_\alpha$ for some $z_\alpha\in\Hom(G,Z(G))$.  Then we may define $z_E,z_{E^c}\in \Hom(G,Z(G))$ by $\pi_E\alpha^* = \pi_E+z_E$ and $\pi_{E^c}\alpha^*=\pi_{E^c}+z_{E^c}$. Equivalently,
  \begin{align}
    z_E=\pi_E z_\alpha,\\
    z_{E^c}=\pi_{E^c}z_\alpha.
  \end{align}.
  Therefore the identities in the summation appearing in \cref{eq:top-EGNO} can be written as
  \begin{align}
    \pi_E(g_j)\inv \delta(t_k) &= \delta(t_l)\delta(y'\pi_E(\gamma^*(\lambda_\chi))z_E(g_j)),\label{eq:top-1}\\
    \pi_{E^c}(h_l)\delta(s_i) &= \delta(s_j)\delta(x'\pi_{E^c}(\gamma^*(\lambda_\eta))\inv z_{E^c}(h_l\inv)).\label{eq:top-2}
  \end{align}

  Since $G$ is purely non-abelian, $\delta\in\Aut(G)$, and in fact $\delta\in\Aut_c(G)$ since $\alpha^*\in\Aut_c(G)$.  Since $z_E,z_{E^c}$ are class functions we can then rewrite \cref{eq:top-EGNO} as
  \begin{align}\label{eq:top-EGNO-2}
    \begin{split}
        (&\ket{s_i,m^\delta}\ot\ket{1,1})\ot(\ket{t_k,n^\delta}\ot\ket{1,1})\\
        &\mapsto
      \beta(\pi_E(\alpha^*(g)\gamma^*(\lambda_\chi))\inv,h) \beta(\pi_{E^c}(\alpha^*(h)\gamma^*(\lambda_\eta)),g)\\
      &\qquad{}\lambda_\eta(\pi_E(\gamma^*(\lambda_\chi))\inv z_E(g)\inv)\lambda_\chi(\pi_{E^c}(\gamma^*(\lambda_\eta)) z_{E^c}(h))\\
    &\qquad{}\Big(\sum_{{\mathclap{\substack{j,l\\\pi_E(g_j\inv)\delta(t_k) = \delta(t_l) \tilde{y}\\\pi_{E^c}(h_l)\delta(s_i)=\delta(s_j) \tilde{x}}}}}
        (\ket{t_l,(\tilde{y}\cdot n)^\delta}\ot \ket{1,1}) \ot ( \ket{s_j, (\tilde{x}\cdot m)^\delta}\ot\ket{1,1})\Big).
    \end{split}
  \end{align}
  This is the same map as \cref{eq:bottom-EGNO} up to a scalar multiple, as desired.

  We next proceed to determine the form $J=(w,\tilde{r},\tilde{p},z)$ must have to make the diagram in \citep[Equation (8.5)]{EGNO} commute.  By \cref{thm:weak-gauge,lem:char-swap} we have
  \begin{gather}
    J((g,\chi),(h,\eta)) = \tilde{r}(h\inv,g)\lambda_\chi(\tilde{p}(\lambda_\eta)\inv z(h))\lambda_\eta(w(g\inv)),\label{eq:J-XY}\\
    J((h,\eta),(g,\chi)) = \tilde{r}(g\inv,h)\lambda_\eta(\tilde{p}(\lambda_\chi)\inv z(g))\lambda_\chi(z(h\inv)),\label{eq:J-YX}\\
    \frac{J((g,\chi),(h,\eta))}{J((h,\eta),(g,\chi))} = \frac{\tilde{r}(g,h) \lambda_\chi((w+z)(h)\tilde{p}^*(\lambda_\eta))}{ \tilde{r}(h,g) \lambda_\eta((w+z)(g)\tilde{p}^*(\lambda_\chi))}.\label{eq:J-rat}
  \end{gather}

  The commutativity of the diagram in \citep[Equation (8.5)]{EGNO} is then equivalent to the equality of
  \begin{align}\label{eq:comm-1a}
  \begin{split}
    \tilde{r}(g,h)&r(h,g)\beta(\pi_{E^c}(\alpha^*(h) \gamma^*(\lambda_\eta)),g)\\
    &\lambda_\chi((w+z+z_{E^c}-z_v)(h)(\tilde{p}^*+\pi_{E^c}\gamma^*)(\lambda_\eta))
  \end{split}
  \end{align}
  and
  \begin{align}\label{eq:comm-2a}
  \begin{split}
    \tilde{r}(h,g)&\beta(\pi_E(\alpha^*(g)\gamma^*(\lambda_\chi)),h)\\
    &\lambda_\eta((w+z+z_{E}-z_u)(g) (\tilde{p}^*-p^*+\pi_E\gamma^*)(\lambda_\chi))
  \end{split}
  \end{align}
  for all $g,h\in G$, $\chi\in\Irr(C_G(g))$, and $\eta\in\Irr(C_G(h))$.

  Taking $g=h=1$ in \cref{eq:comm-1a,eq:comm-2a} and applying \cref{lem:char-swap} we find
  \begin{align*}
    \lambda_\chi((\tilde{p}^*+\pi_{E^c}\gamma^*)(\lambda_\eta)) &= \lambda_\eta((\tilde{p}^*-p^*+\pi_E\gamma^*)(\lambda_\chi))\\
    &= \lambda_\chi((\tilde{p}-p+\gamma\pi_E^*)(\lambda_\eta)).
  \end{align*}
  Since $\chi,\eta\in\Irr(G)$ are arbitrary and the morphisms in question all have central image, this is equivalent to
  \begin{equation}\label{eq:p-rel}
    p+(\tilde{p}^*-\tilde{p}) = \gamma \pi_E^*-\pi_{E^c}\gamma^*.
  \end{equation}

  For the remainder of the proof, we decompose morphisms in terms of their components relation to the decomposition $G=H_1\times H_2$ where $H_1=G_E$ and $H_2=G_{E^c}$.  For example,
  \[ p_{2,1} = \pi_{E^c} p \pi_E^* \in\Hom(\widehat{Z(H_1)},Z(H_2)).\]

  Returning to \cref{eq:p-rel} we find
  \begin{equation}\label{eq:gamma-id}
    \gamma \pi_E^*-\pi_{E^c}\gamma^* = \begin{pmatrix}
      \gamma_{1,1} & 0\\
      \gamma_{2,1}-(\gamma_{1,2})^* & -\gamma_{2,2}^*
    \end{pmatrix}
  \end{equation}
  It will be convenient later if we arrange $\gamma$ to be (upper) triangular.  So by defining
  \begin{align}
  \begin{split}\label{def:tilde-p}
    \tilde{p} &= \begin{pmatrix}
      0 & p_{1,2}\\
      0 & 0
    \end{pmatrix} =\pi_{E} p \pi_{E^c}^*,
  \end{split}\\
  \begin{split}\label{def:gamma}
    \gamma &= \begin{pmatrix}
      p_{1,1} & -p_{1,2}-(p_{2,1})^*\\
      0 & -p_{2,2}^*
    \end{pmatrix}
  \end{split}
  \end{align}
  we have a solution to \cref{eq:p-rel} for any given $p$.

  Now setting $\eta,\chi$ to be the trivial characters in \cref{eq:comm-1a,eq:comm-2a} we have
  \begin{equation*}
    \tilde{r}(g,h)r(h,g) \beta(\pi_{E^c}(\alpha^*(h)),g) = \tilde{r}(h,g) \beta(\pi_E(\alpha^*(g)),h)
  \end{equation*}
  for all $g,h\in G$.  This is equivalent to
  \begin{equation}\label{eq:r-rel}
    r+(\tilde{r}^*-\tilde{r}) = \beta \pi_{E^c}\alpha^* - \alpha \pi_{E}^* \beta^*  .
  \end{equation}
  Writing $\beta,\alpha$ in components in much the same fashion as in the previous case, we find
  \begin{align}\label{eq:beta-id}
  \begin{split}
    \beta&\pi_{E^c}\alpha^* - \alpha \pi_{E}^* \beta^*\\
    &= \begin{pmatrix}
      \beta_{1,2}(\alpha_{2,1})^* - \alpha_{1,1}\beta_{1,1}^* & \beta_{1,2}\alpha_{2,2}^* - \alpha_{1,1} (\beta_{2,1})^*\\
      \beta_{2,2}(\alpha_{2,1})^* - \alpha_{1,2}\beta_{1,1}^* & \beta_{2,2}\alpha_{2,2}^* - \alpha_{2,1}(\beta_{1,2})^*
    \end{pmatrix}.
  \end{split}
  \end{align}

  We observe that the lower left corner in \cref{eq:beta-id} is the only entry which does not involve $\alpha_{E,E}$ or $\alpha_{E^c,E^c}$.  This is problematic as we need to be able to solve for all entries of $\beta$, and the condition $\alpha\in\Aut_c(G)$ forces $\alpha_{E,E},\alpha_{E^c,E^c}$ to be isomorphisms, and also forces that $\alpha_{E,E^c}$ and $\alpha_{E^c,E}$ will not be invertible. The morphism $\tilde{r}$ allows us to fix this, in part by observing that the upper right corner of \cref{eq:beta-id} involves both $\alpha_{E,E}$ and $\alpha_{E^c,E^c}$.  So define
  \begin{equation}\label{eq:r-tilde}
    \tilde{r} = \begin{pmatrix}
      0 & \alpha_{1,1} (\beta_{2,1})^*\\
      0 & 0
    \end{pmatrix}.
  \end{equation}
  Then applying \cref{eq:beta-id} to \cref{eq:r-rel} we have
  \begin{multline}\label{eq:r-beta-id}
    \begin{pmatrix}
      r_{1,1} & r_{1,2}\\
      r_{2,1} & r_{2,2}
    \end{pmatrix}\\ =
    \begin{pmatrix}
      \beta_{1,2}(\alpha_{2,1})^* - \alpha_{1,1} \beta_{1,1}^* & \beta_{1,2}\alpha_{2,2}^*\\
      \beta_{2,2}(\alpha_{1,2})^* - \alpha_{2,1}\beta_{1,1}^* - \beta_{2,1}\alpha_{1,1}^* & \beta_{2,2}\alpha_{2,2}^* - \alpha_{2,1}(\beta_{2,1})^*
    \end{pmatrix}.
  \end{multline}
  Then any $\beta$ satisfying the above relationship for a given $r$ and $\alpha\in\Aut_c(G)$ satisfies
  \begin{align}
    \beta_{1,2} &= r_{1,2}(\alpha_{2,2}^*)\inv,\label{eq:beta-2-1}\\
    \beta_{1,1} &= (\alpha_{1,1}\inv (r_{1,2}(\alpha_{2,2}^*)\inv (\alpha_{2,1})^*-r_{1,1}))^*,\label{eq:beta-1-1}\\
    \beta_{2,2} &= r_{2,2} + \alpha_{2,1} (\beta_{2,1})^*,\label{eq:beta-2-2}
  \end{align}
  \begin{align}\label{eq:beta-1-2}
    \begin{split}
    \beta_{2,1} &= \Big(\beta_{2,2}(\alpha_{1,2})^*-\alpha_{2,1}(\alpha_{1,1}\inv(r_{1,2}(\alpha_{2,2}^*)\inv (\alpha_{2,1})^*\\
    &\qquad{}\qquad{}-r_{1,1})-r_{2,1}) \Big) (\alpha_{1,1}^*)\inv,
    \end{split}
  \end{align}
  The first two identities completely determine $\beta_{1,1}$ and $\beta_{1,2}$ once we know $\alpha$.  Substituting \cref{eq:beta-2-2} into \cref{eq:beta-1-2} we obtain an identity for $\beta_{2,1}$ in terms of itself.  By recursively replacing appearances of $\beta_{2,1}$ we may apply \cref{lem:finite-comp} to obtain a formula for $\beta_{2,1}$ that depends only on $r,\alpha$.  We can then apply this to \cref{eq:beta-2-2} to solve for $\beta_{2,2}$ entirely in terms of $r,\alpha$.  This completely determines $\beta$, provided we can exhibit a suitable choice of $\alpha\in\Aut_c(G)$.

  Next, \cref{eq:r-rel,eq:p-rel} can be used to simplify the equality of \cref{eq:comm-1a,eq:comm-2a} to yield
  \begin{align}\label{eq:comm-simp}
  \begin{split}
    \beta(\pi_{E^c}(\gamma^*(\lambda_\eta)),g)& \lambda_\chi((w+z+z_{E^c}-z_v)(h))\\
    & = \beta(\pi_E(\gamma^*(\lambda_\chi)),h) \lambda_\eta((w+z+z_E-z_u)(g))
  \end{split}
  \end{align}
  for all $g,h\in G$, $\chi\in\Irr(C_G(g))$, and $\eta\in\Irr(C_G(h))$. Note that this identity implicitly involves $\alpha$, as $z_E,z_{E^c}$ are defined in terms of $\alpha$, and furthermore any solution for $\beta$ from the previous case will be defined using $\alpha$.

  Taking $g=1,\eta=\varepsilon$ in \cref{eq:comm-simp} and applying \cref{lem:char-swap} we have
  \begin{align*}
    \lambda_\chi((w+z+z_{E^c}-z_v)(h)) &= \beta(\pi_E(\gamma^*(\lambda_\chi)),h)\\
    &= \lambda_\chi((\gamma \pi_E^*\beta^*)(h)).
  \end{align*}
  Since $\chi,h$ are arbitrary and the terms $\lambda_\chi$ is evaluated at in the above are all central, this is equivalent to
  \begin{equation}\label{eq:wz-rel-1}
    w+z+z_{E^c}-z_v = \gamma\pi_E^*\beta^*.
  \end{equation}
  Similarly, taking $h=1,\chi=\varepsilon$ in \cref{eq:comm-simp} we find
  \begin{equation}\label{eq:wz-rel-2}
    w+z+z_E-z_u = \gamma \pi_{E^c}^*\beta^*
  \end{equation}
  We need only solve these two equations, consistently with the previous cases, to completely decide the equality of \cref{eq:comm-1a,eq:comm-2a}.

  From \cref{def:gamma} we have
  \begin{align*}
    \gamma \pi_E^* &= \begin{pmatrix}
      p_{1,1}&0\\
      0&0
    \end{pmatrix},\\
  \gamma \pi_{E^c}^* &= \begin{pmatrix}
     0 & -p_{1,2}-(p_{2,1})^*\\
     0 & -p_{2,2}^*
  \end{pmatrix}.
  \end{align*}
  Therefore
  \begin{gather}
    \gamma\pi_E^*\beta^* = \begin{pmatrix}
      p_{1,1} \beta_{1,1}^* & p_{1,1}(\beta_{2,1})^*\\
      0 & 0
    \end{pmatrix},\label{eq:gamma-beta-1}\\
    \gamma\pi_{E^c}^*\beta^* = \begin{pmatrix}
      -(p_{1,2}+(p_{2,1})^*)(\beta_{1,2})^* & -(p_{1,2}+(p_{2,1})^*)\beta_{2,2}^*\\
      -p_{2,2}^*(\beta_{1,2})^* & -(p_{2,1})^*\beta_{2,2}^*
    \end{pmatrix}.\label{eq:gamma-beta-2}
  \end{gather}

  We define and write morphisms $f,\tilde{v},\tilde{u},\tilde{\alpha}=\alpha^*-1\in\Hom(G,Z(G))$ by
  \begin{align*}
    f &= z+w,\\
    z_v &= \begin{pmatrix}
      \tilde{v}_{1,1} & \tilde{v}_{1,2}\\
      \tilde{v}_{2,1} & \tilde{v}_{2,2}
    \end{pmatrix},\\
    z_u &= \begin{pmatrix}
      \tilde{u}_{1,1} & \tilde{u}_{1,2}\\
      \tilde{u}_{2,1} & \tilde{u}_{2,2}
    \end{pmatrix},\\
    z_E &= \begin{pmatrix}
      \tilde{\alpha}_{1,1} & \tilde{\alpha}_{1,2}\\
      0 & 0
    \end{pmatrix},\\
    z_{E^c} &= \begin{pmatrix}
      0 & 0\\
      \tilde{\alpha}_{2,1} & \tilde{\alpha}_{2,2}
    \end{pmatrix}.
  \end{align*}
  Note that $\tilde{\alpha}_{2,1}=(\alpha_{1,2})^*$ and $\tilde{\alpha}_{1,2}=(\alpha_{2,1})^*$.  \Cref{eq:wz-rel-1,eq:wz-rel-2} are then equivalent to the following eight identities
  \begin{align}
    \tilde{\alpha}_{2,1} &= \tilde{v}_{2,1}-f_{2,1},\label{eq:t-alpha-12}\\
    \tilde{\alpha}_{2,2} &= \tilde{v}_{2,2}-f_{2,2},\label{eq:t-alpha-22}\\
    \tilde{\alpha}_{1,1} &= \tilde{u}_{1,1}-f_{1,1} -(p_{1,2}+(p_{2,1})^*)(\beta_{1,2})^*,\label{eq:t-alpha-11}\\
    \tilde{\alpha}_{1,2} &= \tilde{u}_{1,2}-f_{1,2} -(p_{1,2}+(p_{2,1})^*)\beta_{2,2}^*,\label{eq:t-alpha-21}\\
    f_{1,1} &= p_{1,1}\beta_{1,1}^*+\tilde{v}_{1,1},\label{eq:f-11}\\
    f_{1,2} &= p_{1,1}(\beta_{2,1})^*+\tilde{v}_{1,2},\label{eq:f-21}\\
    f_{2,1} &= \tilde{u}_{2,1}-p_{2,2}^*(\beta_{1,2})^*,\label{eq:f-12}\\
    f_{2,2} &= \tilde{u}_{2,2}-(p_{2,1})^*\beta_{2,2}^*.\label{eq:f-22}
  \end{align}
  Observe that any solution to these equations indeed forces $\tilde{\alpha}\in\Hom(G,Z(G))$, and so by \cref{lem:Adney} $\alpha^*=1+\tilde{\alpha}\in\Aut_c(G)$, which is consistent with our assumptions on $\alpha$ in previous steps.  So we need only show that there are $\tilde{\alpha},f\in\Hom(G,Z(G))$ satisfying these equations.

  Now by \cref{eq:t-alpha-12,eq:f-12,eq:beta-2-1} we have
  \begin{align}\label{eq:t-alpha-12-a}
  \begin{split}
    \tilde{\alpha}_{2,1} &= \tilde{u}_{2,1}-\tilde{v}_{2,1}+p_{2,2}^*(\beta_{1,2})^*\\
    &= \tilde{u}_{2,1}-\tilde{v}_{2,1}+p_{2,2}^* \alpha_{2,2}\inv r_{2,1}\\
    &= \tilde{u}_{2,1}-\tilde{v}_{2,1}+p_{2,2}^*(1+\tilde{\alpha}_{2,2}^*)\inv r_{2,1}.
  \end{split}
  \end{align}
  Thus $\tilde{\alpha}_{2,1}$ is solved in terms of the unknown $\tilde{\alpha}_{2,2}$, and the known morphisms $\tilde{u},\tilde{v},p,r$.

  Similarly, \cref{eq:t-alpha-21,eq:f-21,eq:beta-1-2,eq:t-alpha-12-a,lem:recursion} solves for $\tilde{\alpha}_{1,2}$ in term of the unknown morphisms $\tilde{\alpha}_{2,2},\tilde{\alpha}_{1,1}$.  We can then use  \cref{eq:t-alpha-11,eq:f-11,eq:beta-1-1,lem:recursion} to solve for $\tilde{\alpha}_{1,1}$ in terms of the unknown morphism $\tilde{\alpha}_{2,2}$.  Then we can use \cref{eq:t-alpha-22,eq:f-22,eq:beta-2-2,lem:recursion} to solve for $\tilde{\alpha}_{2,2}$ entirely in terms of the known morphisms.  Substituting this back into the previous expressions allows us to solve for the remaining components of $\tilde{\alpha}$ entirely in terms of the known morphisms.  The morphism $f$ is then defined by \cref{eq:f-11,eq:f-21,eq:f-12,eq:f-22}, and we are free to pick any $z,w\in\Hom(G,Z(G))$ such that $f=z+w$ (such as $z=f$ and $w=0$).

  All combined we have found choices of $\psi,J$ making $(\Psi,J)$ into a braided tensor equivalence $\Rep(\D(G),R)\to\Rep(\D(G),R_E)$, as desired. This completes the proof.
\end{proof}
Indeed, the proof is constructive, and in principle with \cref{lem:recursion} the example equivalence can be made fully explicit in a finite number of steps.  The proof needed that $G$ was purely non-abelian to ensure that we could solve for the components of $\tilde{\alpha}$, and one component of $\beta$, by a finite number of successive replacements.  The proof otherwise works for a more arbitrary group provided we take care that the required iterative replacement arguments remain valid, and that $\alpha^*$ remains a central automorphism.

\begin{example}
  In the special case when $R$ is left-handed, meaning $E^c=\emptyset$, the equations for $\tilde{\alpha}$ reduce to
  \begin{equation*}
    \alpha^* - 1 = z_R + p\alpha\inv r,
  \end{equation*}
  When $z_R$ is trivial, this equation appears in \cref{lem:pna}.  Deciding when this equation admitted solutions, and so the special case where $E^c=\emptyset$, was the motivation for how to solve the general case. The case $E=\emptyset$, when $R$ is right-handed, is similar.
\end{example}

\begin{thm}
  For any group $G$ and $E\subseteq\{1,...,n\}$ there is a braided tensor equivalence $\Rep(\D(G),R_E)\to\Rep(\D(G),R_0)$.
\end{thm}
\begin{proof}
  We have a braided tensor equivalence $\Rep(\D(G))\cong \Z(\VecG_G)$.  Furthermore, $\VecG_G$ is categorically Morita equivalent to $\Rep(G)$, which is equivalent to there being a braided tensor equivalence $\Z(\VecG_G)\cong \Z(\Rep(G))$.  Now $\Rep(G)$ admits a symmetric braiding, so there is also a braided tensor equivalence $\Z(\Rep(G)^{\op})\cong \Z(\Rep(G))$.  We always have a braided tensor equivalence $\Z(\C^{\op})\cong\Z(\C)^{\text{rev}}$. Therefore we have braided tensor equivalences $\Rep(\D(G),R_0)\cong \Rep(\D(G),R_0)^{\text{rev}} \cong \Rep(\D(G),R_1)$.

  More generally we have that $\D(G)=\D(G_E)\ot \D(G_{E^c})$, and there is an isomorphism of quasitriangular Hopf algebras $(\D(G),R_E)\to (\D(G_E),R_0)\ot (\D(G_{E^c}),R_1)$.  This isomorphism therefore yields a braided tensor equivalence $\Rep(\D(G),R_E) \cong \Rep(\D(G_E),R_0)\boxtimes \Rep(\D(G_{E^c}),R_1)$. The desired result then follows by applying the previous case.
\end{proof}

\begin{cor}\label{cor:one-structure}
  A group $G$ is purely non-abelian if and only if the pivotal category $\Z(\VecG_G)$ admits exactly one structure of a braided category, up to braided tensor equivalence.
\end{cor}

\section{Application to Frobenius-Schur indicators}\label{sec:indicators}
We now show how to use braid gaugings to obtain new identities and relations for the indicators.  We consider only purely non-abelian groups when discussing $\Rep(\D(G))$ in this section, so that every braid gauging preserves the modularity property.

Fix a modular fusion category $\C$ with braiding $c$, basis $\{X_1,...,X_m\}$, and $\Gamma=\{1,...,m\}$ the label set.  We let $N_{i,k}^j = \dim \C(X_i\ot X_k,X_j)$ for all $i,j,k\in\Gamma$ denote the fusion rules.  We also let $d_i$ be the pivotal dimension of $X_i$ for all $i\in\Gamma$.  The fusion rules and pivotal dimensions do not depend on the braiding, and so are independent of braid gauging.  We let $T,S$ denote the modular data for $\C$, as usual.

The indicators can be computed in terms of the fusion rules and the modular data via the following result, which we call the Bantay-Ng-Schauenburg (BNS) formula.
\begin{thm}[{\citep[Theorem 7.5]{NS07a}}]\label{thm:BNS}
    Let $\C$ be a modular fusion category. For any $j\in\Gamma$ and $m\in\BN$ we have
    \[ \nu_m(X_j) = \frac{1}{\dim(\C)} \sum_{i,k\in\Gamma} N_{i,k}^j d_i d_k \Big(\frac{T_i}{T_k}\Big)^m.\]
\end{thm}
As the categorical dimension $\dim(\C)$ is also independent of the braiding, and the values $\nu_m(X_j)$ depend only on the pivotal structure, the only terms in the BNS formula that depend on the braiding are the $T$-matrix entries.  By varying the braiding, such as by braid gauging, we can therefore take linear combinations of BNS formulas to obtain new formulas and dependencies.

\begin{thm}\label{thm:indicator-gauge}
    Let $\C$ be a braided spherical fusion category, and suppose $b$ is a braid gauging such that $\cyc{b}$ is modular.  Then for any $j\in\Gamma$, $m\in\BN$, and $\mu\in U(1)$
    \begin{equation*}
        \sum_{\substack{i,k\in \Gamma\\ b(X_i,X_i)^m = \mu b(X_k,X_k)^m}} N_{i,k}^j d_i d_k \Big( \frac{T_i}{T_k}\Big)^m = \delta_{1,\mu} \dim(\C)\nu_m(X_j)
    \end{equation*}
\end{thm}
\begin{proof}
  Let notation and assumptions be as in the statement. Since the braid gaugings are a group, for every $s\in\BZ$ we can apply \cref{thm:BNS} and \cref{thm:mod-gauge} using the $T$-matrix $T^{b^s}$ obtained from gauging $c$ by $b^s$ to obtain
  \begin{equation}\label{eq:BNS-1}
    \nu_m(X_j) = \frac{1}{\dim(\C)} \sum_{i,k\in\Gamma} N_{i,k}^j d_i d_k \Big(\frac{b^{sm}(X_k,X_k)}{b^{sm}(X_i,X_i)}\Big)\Big(\frac{T_i}{T_k}\Big)^m.
  \end{equation}

  Since $\U(\C)$ is finite, we let $\varpi$ be any primitive $\exp(\Img(b))$-th root of unity. Then we can rewrite \cref{eq:BNS-1} as
  \begin{equation}\label{eq:BNS-2}
    \nu_m(X_j) = \frac{1}{\dim(\C)} \sum_{t=0}^{\exp(\Img(b))-1} \sum_{\substack{i,k\in\Gamma\\b(X_k,X_k)^m = \varpi^{t} b(X_i,X_i)^m}} N_{i,k}^j d_i d_k  \varpi^{ts}\Big(\frac{T_i}{T_k}\Big)^m.
  \end{equation}
  Summing these equations over all $0\leq s< \exp(\Img(b))$ then gives the desired equation in the special case $\mu=1$.  Taking linear combinations of \cref{eq:BNS-2} using powers of $\varpi$ gives the desired equation when $\mu$ is any power of $\varpi$. For any other value of $\mu$ the identity $\mu b(X_i,X_i)^m = b(X_k,X_k)^m$ cannot hold for any $i,k\in\Gamma$, so the desired equation is trivially true.

  This completes the proof.
\end{proof}

We demonstrate a number of uses (and non-uses) for this using $\C=\Rep(\D(G))$ for various choices of purely non-abelian groups $G$ and braid gaugings.

\begin{example}
  When $G$ is indecomposable, centerless, and perfect then $R_0$ and $R_1$ are the only quasitriangular structures of $\D(G)$.  If we remove the indecomposable property, then the decomposition into indecomposables is unique (up to the ordering of factors) and the only quasitriangular structures are the standard $E$-chiral structures. In particular, there are no non-trivial braid gaugings.
\end{example}
\begin{example}\label{ex:inverse-form}
By subtracting and adding the indicator formulas obtained from $R_0$ and $R_1$, we have
  \begin{gather*}
    |G|^2\nu_m(X_j) = \sum_{i,k\in\Gamma} N_{i,k}^j d_i d_k \Real((T_i/T_k)^m),\\
    0 = \sum_{i,k\in\Gamma} N_{i,k}^j d_i d_k \operatorname{Im}((T_i/T_k)^m).
  \end{gather*}
\end{example}

\begin{example}\label{ex:Sn-bichar}
  Let $G=S_k$ for $k\geq 3$.  Let $r\in\BCh{S_k}\cong\BZ_2$ be the non-trivial bicharacter given by $r(g,h)=-1$ if both $g,h$ are odd permutations and $r(g,h)=1$ otherwise.  Let $3\leq m\in\BN$ be odd. As before, write $X_i=(t_i,\phi_i)$.  From the gauging $b=(0,r,0,0)$ and \cref{thm:indicator-gauge} we have
  \begin{align}\label{eq:Sn-odd-indic}
    \nu_m(X_j) &= \frac{1}{|G|^2} \sum_{\substack{i,l\in\Gamma\\\sgn(t_i t_l)=1}} N_{il}^j d_i d_l \Big(\frac{T_i}{T_l}\Big)^m,
  \end{align}
  and
  \begin{align*}
    \sum_{\substack{i,l\in\Gamma\\\sgn(t_i t_l)=-1}} N_{il}^j d_i d_l \Big(\frac{T_i}{T_l}\Big)^m = 0.
  \end{align*}
  We therefore obtain new information regarding the fusion rules between irreducibles with different $\sgn$ values.

  If instead we had $m$ even, then since $b(X,X)^2=1$ for all $X$ we obtain no new information.
\end{example}

The previous example gives an alternative proof to the following result of Courter.

\begin{prop}[{\citep[Proposition 3.4.1]{Co}}]
  Let $(g,\chi)$ be an irreducible $\D(S_n)$-module.  If $g$ is an odd permutation and $m\in\BN$ is odd then $\nu_m(g,\chi)=0$.
\end{prop}
\begin{proof}
  By \cref{ex:Sn-bichar} $\nu_m(g,\chi)$ can be computed using only those irreducibles $X_i,X_l$ such that $t_i,t_l$ have the same sign as permutations.  Furthermore, if $(g,\chi)$ is a submodule of $(a,\alpha)\ot (b,\beta)$ then $g=xy$ for some $x\in\cl(a)$ and $y\in\cl(b)$.  But if $a,b$ have the same sign this implies that $g$ is an even permutation, a contradiction.  Therefore every fusion coefficient appearing in \cref{eq:Sn-odd-indic} is zero, and thus $\nu_m(g,\chi)=0$ as desired.
\end{proof}

\begin{example}\label{ex:bichar-sign}
  Generalizing \cref{ex:Sn-bichar}, let \[R=\cmorph[1][r][0][0]=(0,r,0,0)R_0\neq R_0\] be a quasitriangular structure of $\D(G)$.  Then $b((g,\chi),(g,\chi))=\overline{r(g,g)}$ for all irreducible $\D(G)$-modules $(g,\chi)$.  By writing $X_i=(t_i,\phi_i)$, we therefore have
  \begin{equation*}
       \delta_{1,\mu} \nu_m(X_j)= \frac{1}{|G|^2} \sum_{\substack{i,k\in \Gamma\\ r(t_k,t_k)^m = \mu r(t_i,t_i)^m}} N_{i,k}^j d_i d_k \Big( \frac{T_i}{T_k}\Big)^m
  \end{equation*}
  for all $\mu\in U(1)$.
\end{example}

\begin{example}\label{ex:semidirect}
  Let $G=\cyc{a,b \ | \ a^p=b^{16}=1, ba=a\inv b}=\BZ_p\rtimes\BZ_{16}$.  Then $Z(G)=\cyc{b^2}\cong\BZ_8$ and $G'=\cyc{a}\cong\BZ_p$.  Let $\alpha$ be any generator of $\widehat{G}\cong\BZ_{16}$.  We then have that $\alpha(b)=\mu_{16}$ is a primitive 16-th root of unity.  We next define $r\in\BCh{G}$ by $r(a^i b^j)=\alpha^j$.

  Write
  \begin{gather*}
    X_i = (a^{i_1}b^{i_2},\chi),\\
    X_j = (a^{j_1}b^{j_2},\beta),\\
    X_k = (a^{k_1}b^{k_2},\eta).
  \end{gather*}
  Fix the gauging $b=(0,0,r,0)$ (of $R_0$).  Then we have $b(X_i,X_i) = \mu_{16}^{-i_2^2}$.  From \cref{ex:bichar-sign} we conclude that
  \begin{equation*}
    \sum_{{\mathclap{\substack{i,k\in\Gamma\\ \mu\cdot \mu_{16}^{m(i_2)^2} = \mu_{16}^{m(k_2)^2}}}}} N_{i,k}^j d_i d_k \Big(\frac{T_i}{T_k}\Big)^m = 256p^2\delta_{1,\mu}\nu_m(X_j)
  \end{equation*}
  for all $\mu\in U(1)$.  For $\mu=1$ the sum is over those labels $i,k\in\Gamma$ with \[i_2^2\equiv k_2^2 \bmod \frac{16}{\gcd(16,m)}.\]
\end{example}
\begin{example}\label{ex:p-sign}
  Let $R=\cmorph[1][0][p][0]=(0,0,p,0)R_0\neq R_0$ be a quasitriangular structure of $\D(G)$.  We then have $b((g,\chi),(g,\chi))= \lambda_\chi(p^*(\lambda_\chi)\inv)$ for all simple objects $(g,\chi)$.  Therefore
  \begin{equation*}
       \sum_{{\mathclap{\substack{i,k\in \Gamma\\ \lambda_{\phi_k}(p^*(\lambda_{\phi_k}))^m = \mu \lambda_{\phi_i}(p^*(\lambda_{\phi_i}))^m}}}}
        N_{i,k}^j d_i d_k \Big( \frac{T_i}{T_k}\Big)^m =\delta_{1,\mu} |G|^2 \nu_m(X_j).
  \end{equation*}
\end{example}
\begin{example}\label{ex:z-sign}
    Let $R=\cmorph[1][0][0][z]=(0,0,0,z)R_0\neq R_0$ be a quasitriangular structure of $\D(G)$.  Then $b((g,\chi),(g,\chi)) = \lambda_\chi(z(g))$.  Therefore
    \begin{equation*}
       \sum_{{\mathclap{\substack{i,k\in \Gamma\\ \lambda_{\phi_i}(z(t_i^m)) = \mu \lambda_{\phi_k}(z(t_k^m))}}}}
        N_{i,k}^j d_i d_k \Big( \frac{T_i}{T_k}\Big)^m = \delta_{1,\mu}|G|^2 \nu_m(X_j).
  \end{equation*}
\end{example}

\section{Application to the Verlinde formula}\label{sec:verlinde}
We now give an application similar to those in the preceding section, but using the $S$-matrix and the Verlinde formula for the fusion rules~\citep{BakKir:book}. We consider only purely non-abelian groups when discussing $\Rep(\D(G))$ in this section, so that every braid gauging preserves the modularity property.

As before, let $\C$ be a modular fusion category with braiding $c$, basis $\B=\{X_1,...,X_m\}$, and label set $\Gamma=\{1,...,m\}$.  Let $0$ denote both the (isomorphism class of) the identity object as well as its label in $\Gamma$. For $j\in\Gamma$ we define $j^*\in \Gamma$ by $X_j^*\cong X_{j^*}$. Then the Verlinde formula is
\begin{equation}\label{eq:verlinde}
N_{i,k}^{j^*} = \frac{1}{\dim(\C)} \sum_{a\in\Gamma} \frac{ S_{i,a} S_{k,a} S_{j,a}}{S_{0,a}}.
\end{equation}
Note that $0^*=0\in\Gamma$ under these definitions.

As the fusion rules are invariant under braid gauging, we can take linear combinations of \cref{eq:verlinde} over various gaugings that also yield a modular category to obtain new identities and dependencies.

\begin{thm}\label{thm:fusion-gauge}
  Let $\C$ be a modular fusion category, and let $b$ be a braid gauging such $\cyc{b}$ is modular.  Then for all $X,Y,Z\in\B$ and $\mu\in U(1)$ we have
  \begin{equation*}
    \sum_{{\mathclap{\substack{A\in\B\\b(X\ot Y\ot Z,A)=\mu b(A^*,X\ot Y\ot Z)}}}} \frac{ S_{X,A} S_{Y,A} S_{Z,A}}{S_{0,A}} = \dim(\C)\delta_{1,\mu} N_{X,Y}^{Z^*}
  \end{equation*}
\end{thm}
\begin{proof}
  Let notation and assumptions be as in the statement.  Then by applying \cref{thm:mod-gauge} to the Verlinde formula obtained from gauging by $b^s$ for some $s\in\BZ$, and using the properties of $b$ and the (hidden) degree function $\deg$, we have
  \begin{equation}\label{eq:verlinde-1}
    N_{X,Y}^{Z^*} = \frac{1}{\dim(\C)} \sum_{A\in\B} \left(b(X\ot Y\ot Z,A) b(A,X\ot Y\ot Z)\right)^s \frac{ S_{X,A} S_{Y,A} S_{Z,A}}{S_{0,A}}.
  \end{equation}
  As in the proof of \cref{thm:indicator-gauge}, we let $\varpi$ be a primitive $\exp(\Img(b))$-th root of unity.  Then we may write \cref{eq:verlinde-1} as
  \begin{equation}\label{eq:verlinde-2}
    N_{X,Y}^{Z^*} = \frac{1}{\dim(\C)} \sum_{t=0}^{\exp(\Img(b))-1} \sum_{{\mathclap{\substack{A\in\B\\ b(X\ot Y\ot Z,A)=\varpi^{t} b(A^*,X\ot Y\ot Z)}}}} (\varpi)^{st} \frac{ S_{X,A} S_{Y,A} S_{Z,A}}{S_{0,A}}.
  \end{equation}
  Summing over $s$ gives the desired equation in the special case $\mu=1$.  Taking linear combinations gives the desired equation when $\mu$ is a power of $\varpi$.  For all other $\mu$, there are no $A\in\B$ with $b(X\ot Y\ot Z,A)=\mu b(A^*,X\ot Y\ot Z)$, and so the desired identity is trivially true.

  This completes the proof.
\end{proof}

We conclude with a few examples using $\Rep(\D(G))$.
\begin{example}\label{ex:bichar-S}
  Let $R=\cmorph[1][r][0][0]=(0,r,0,0)R_0$.  Taking $X=(g,\chi), Y=(h,\eta), Z=(l,\kappa)$ then we have
  \[ \sum_{{\mathclap{\substack{A=(t,\phi)\in\B\\r(t,ghl)r(ghl,t)=1}}}} \frac{ S_{X,A} S_{Y,A} S_{Z,A}}{S_{0,A}} = |G|^2 N_{X,Y}^{Z^*},\]
  and for $\mu\neq 1$
  \[ \sum_{{\mathclap{\substack{(t,\phi)\in\B\\r(t,ghl)r(ghl,t)=\mu}}}} \frac{ S_{X,A} S_{Y,A} S_{Z,A}}{S_{0,A}}=0,\]
\end{example}
\begin{example}
  Continuing the previous example, when $G=S_k$ and $r$ is the unique non-trivial bicharacter, then $r(t,ghl)r(ghl,t)=1$ holds for all $a\in\Gamma$ and $g,h,l\in G$. So in this case no terms are eliminated from the Verlinde formula for any choice of simples $X,Y,Z$.
\end{example}
\begin{example}
  On the other hand, consider $G=\BZ_p\rtimes\BZ_{16}$, with notation and $r$ as given in \cref{ex:semidirect}. Write
  \begin{gather*}
    X_i = (a^{i_1}b^{i_2},\chi)\\
    X_j = (a^{j_1}b^{j_2},\beta)\\
    X_k = (a^{k_1}b^{k_2},\eta)\\
    X_s = (a^{s_1}b^{s_2},\alpha),
  \end{gather*}
  where to avoid conflict with the element $a\in G$ we have used the label $s\in\Gamma$ instead of $a\in\Gamma$.  Then by \cref{ex:bichar-S} we have that
  \[ \sum_{{\mathclap{\substack{s\in\Gamma\\s_2(i_2+j_2+k_2)\equiv 0\bmod 8}}}}\frac{ S_{i,s} S_{k,s} S_{j,s}}{S_{0,s}} = 256p^2 N_{i,k}^{j^*}.\]
  So, as a special case, if $i_2+j_2+k_2$ is odd then it follows that the sum is run over those $X_s$ with $s_2\equiv 0\bmod 8$.  At the other extreme, the sum is precisely the usual Verlinde formula if and only if $i_2+j_2+k_2\equiv 0\bmod 8$.
\end{example}

Lastly, we note that while the examples of these last two sections have focused on quasitriangular structures of the form $\cmorph[1][r][p][z]$, they can in fact be expressed for more general quasitriangular structures (including ones that are not left-handed) in the obvious fashion.  The particular case of replacing $R$ with the inverse braiding $\tau(R\inv)$ simply conjugates the BNS and Verlinde formulas, allowing us to get formulas using real and imaginary parts as in \cref{ex:inverse-form}.



\bibliographystyle{plainnat}
\bibliography{../references,../refs-quasitriangular}
\end{document}